\documentclass[11pt,times,letter]{article}

\usepackage[margin=1in]{geometry} 
\usepackage{parskip}
\usepackage{longtable,color,natbib}
\usepackage{booktabs} 
\usepackage{array} 
\usepackage{paralist} 
\usepackage{verbatim} 
\usepackage{subfigure} 
\usepackage{appendix}
\usepackage{amssymb}
\usepackage{amsthm}
\usepackage{amsmath}

\definecolor{linkcolor}{rgb}{0.1,0,0.7}
\definecolor{urlcolor}{rgb}{1,0,0}
\usepackage[unicode=true,colorlinks,citecolor=linkcolor,linkcolor=linkcolor,urlcolor=blue]{hyperref}
\usepackage{graphicx} 
\usepackage{epstopdf}
\usepackage{bm}

\newcommand{\Px}{ \mathbb{P} }

\newcommand{\Ex}{ \mathbb{E} }

\def\esssup_#1{\underset{#1}{\mathrm{ess\,sup\, }}}
\def\essinf_#1{\underset{#1}{\mathrm{ess\,inf\, }}}
\def\argmax_#1{\underset{#1}{\mathrm{arg\,max\, }}}
\def\argmin_#1{\underset{#1}{\mathrm{arg\,min\, }}}

\newcommand{\Fx}{\mathbb{F} }

\newcommand{\R}{\mathbb{R}}

\newtheorem{theorem}{Theorem}[section]

\numberwithin{equation}{section}

\newtheorem{proposition}[theorem]{Proposition}
\newtheorem{remark}[theorem]{Remark}
\newtheorem{lemma}[theorem]{Lemma}

\allowdisplaybreaks[4]

\title{Stochastic control problems with state-reflections arising from relaxed benchmark tracking}

\author{Lijun Bo \thanks{Email: lijunbo@ustc.edu.cn, School of Mathematics and Statistics, Xidian University, Xi'an, 710126, China.}
\and
Yijie Huang \thanks{Email: huang1@mail.ustc.edu.cn, School of Mathematical Sciences, University of Science and Technology of China, Hefei, 230026, China}
\and
Xiang Yu \thanks{Email: xiang.yu@polyu.edu.hk, Department of Applied Mathematics, The Hong Kong Polytechnic University, Hung Hom, Kowloon, Hong Kong.}
}
\date{\vspace{-0.5cm}}

\begin{document}

\maketitle

\begin{abstract}
This paper studies stochastic control problems motivated by optimal consumption with wealth benchmark tracking. The benchmark process is modeled by a combination of a geometric Brownian motion and a running maximum process, indicating its increasing trend in the long run. We consider a relaxed tracking formulation such that the wealth compensated by the injected capital always dominates the benchmark process. The stochastic control problem is to maximize the expected utility of consumption deducted by the cost of the capital injection under the dynamic floor constraint. By introducing two auxiliary state processes with reflections, an equivalent auxiliary control problem is formulated and studied, which leads to the HJB equation with two Neumann boundary conditions. We establish the existence of a unique classical solution to the dual PDE using some novel probabilistic representations involving the local time of some dual processes together with a tailor-made decomposition-homogenization technique. The proof of the verification theorem on the optimal feedback control can be carried out by some stochastic flow analysis and technical estimations of the optimal control.

\vspace{0.1in}

\vspace{0.1in}

\end{abstract}

\vspace{0.2in}

\hypertarget{introduction}{
\section{Introduction}\label{introduction}}
The continuous time optimal portfolio-consumption problem has been extensively studied in different models since the seminal work \cite{Merton69} and \cite{Merton1971}. In the present paper, we aim to study this problem from a new perspective by simultaneously considering the wealth tracking with respect to an exogenous benchmark process. Similar to a large body of literature on optimal tracking portfolio, see, for instance, \cite{Browne99a},  \cite{Browne99b}, \cite{Browne00},  \cite{Tepla01}, \cite{Gaivoronski05}, \cite{YaoZZ06}, \cite{Strub18}, the goal of tracking is to ensure the agent's wealth level being close to a targeted benchmark such as the market index, the inflation rate or the consumption index. However, unlike the conventional formulation of optimal tracking portfolio in the aforementioned studies, we adopt the relaxed tracking formulation proposed in \cite{BoLiaoYu21} using capital injection such that the benchmark process is regarded as a minimum floor constraint of the total capital.

Let $(\Omega, \mathcal{F}, \Fx,\mathbb{P})$ be a filtered probability space with the filtration $\mathbb{F}=(\mathcal{F}_t)_{t\geq 0}$ satisfying the usual conditions, which supports a $d$-dimensional Brownian motion $(W^1,\ldots,W^d)=(W_t^1,\ldots,W^d_t)_{t\geq 0}$. We consider a market model consisting of $d$ risky assets, whose price dynamics are described by
\begin{align}
\frac{dS_t^i}{S_t^i}= \mu_i dt+\sum_{j=1}^{d} \sigma_{ij} dW_t^j,\quad i=1,\ldots,d\label{stockSDE}
\end{align}
with the return rate $\mu_i\in\R$, $i=1,\ldots,d$, and the volatility $\sigma_{ij}\in\R$, $i,j=1,\ldots,d$. Let us denote $\mu:=(\mu_1,\ldots,\mu_d)^{\top}$ with $\top$ representing the transpose operator, and $ \sigma:=(\sigma_{ij})_{d\times d}$.  It is assumed that $\sigma$ is invertible. We also assume that the riskless interest rate $r=0$ that amounts to the change of num\'{e}raire and $ \mu$ is not a zero vector. From this point onwards, all values are defined after the change of num\'{e}raire. At time $t\geq 0$, let $\theta_t^i$ be the amount of wealth that a fund manager allocates in asset $S^i=(S^i_t)_{t\geq 0}$ and $c_t$ be the consumption rate. The self-financing wealth process of the agent satisfies the controlled SDE:
\begin{align}\label{eq:wealth2}
V^{\theta,c}_t &=\textrm{v}+\int_0^t\theta_s^{\top}\mu ds+\int_0^t\theta_s^{\top}\sigma dW_s-\int_0^tc_sds,\quad \forall t\geq 0,
\end{align}
where $\textrm{v}\geq0$ represents the initial wealth level of the agent.

To incorporate the wealth tracking into our optimal consumption problem, let us consider a general type of benchmark processes $M=(M_t)_{t\geq0}$, which is described by
\begin{align}\label{eq:Mtn}
M_t=m_t+Z_t,\quad \forall t\geq 0,
\end{align}
where $Z_t=z+\int_0^t\mu_Z Z_sds+ \int_0^t\sigma_Z Z_sdW^{\eta}_s$ is a GBM and $m_t:=\max\{\textrm{m}, \sup_{s\leq t}B_s\}$ is the running maximum process of the drifted Brownian motion $B_t=b+\mu_B t+ \sigma_B W_t^{\gamma}$. Here, the model parameters $z,{\rm m}\geq0$, $b\in\R$, $\mu_Z,\mu_B\in\R$ and $\sigma_Z,\sigma_B\geq0$. For the correlative vector $\gamma=(\gamma_1,\ldots,\gamma_d)^{\top}\in[-1,1]^d$, the process $W^{\gamma}=(W_t^{\gamma})_{t\geq 0}$ is a linear combination of the $d$-dimensional Brownian motion $(W^1,\ldots,W^d)$ with weights $\gamma$, which itself is a Brownian motion. Similarly, the process $W^{\eta}=(W_t^{\eta})_{t\geq 0}$ is a linear combination of the $d$-dimensional Brownian motion $(W^1,\ldots,W^d)$ with weights $\eta=(\eta_1,\ldots,\eta_d)\in[-1,1]^d$. 

The benchmark process in the general form of \eqref{eq:Mtn} can effectively capture the long-term increasing trend of many typical benchmark processes, such as S\&P~500, NASDAQ and Dow Jones, or the movements of CPI index and higher education costs in the long run. Figure \ref{fig:index}-(a) illustrates the increasing trend of simulated sample paths of \eqref{eq:Mtn}, which is consistent to Figure \ref{fig:index}-(b) that displays the long term growing trend of the observed data of S\&P500, NASDAQ and Dow Jones from April 1, 2010 to November 01, 2020. Similarly, Figure \ref{fig:index}-(c) plots the Consumer Price Index for Urban Wage Earners and Clerical Workers (CPI-W) from 1984 to 2023, and Figure \ref{fig:index}-(d) plots the total cost of U.S. undergraduate students over time from 1963 to 2021, which both exhibit the same increasing trend in the long run.

\begin{figure}[]
\centering
 \subfigure[]{
        \includegraphics[width=7.6cm]{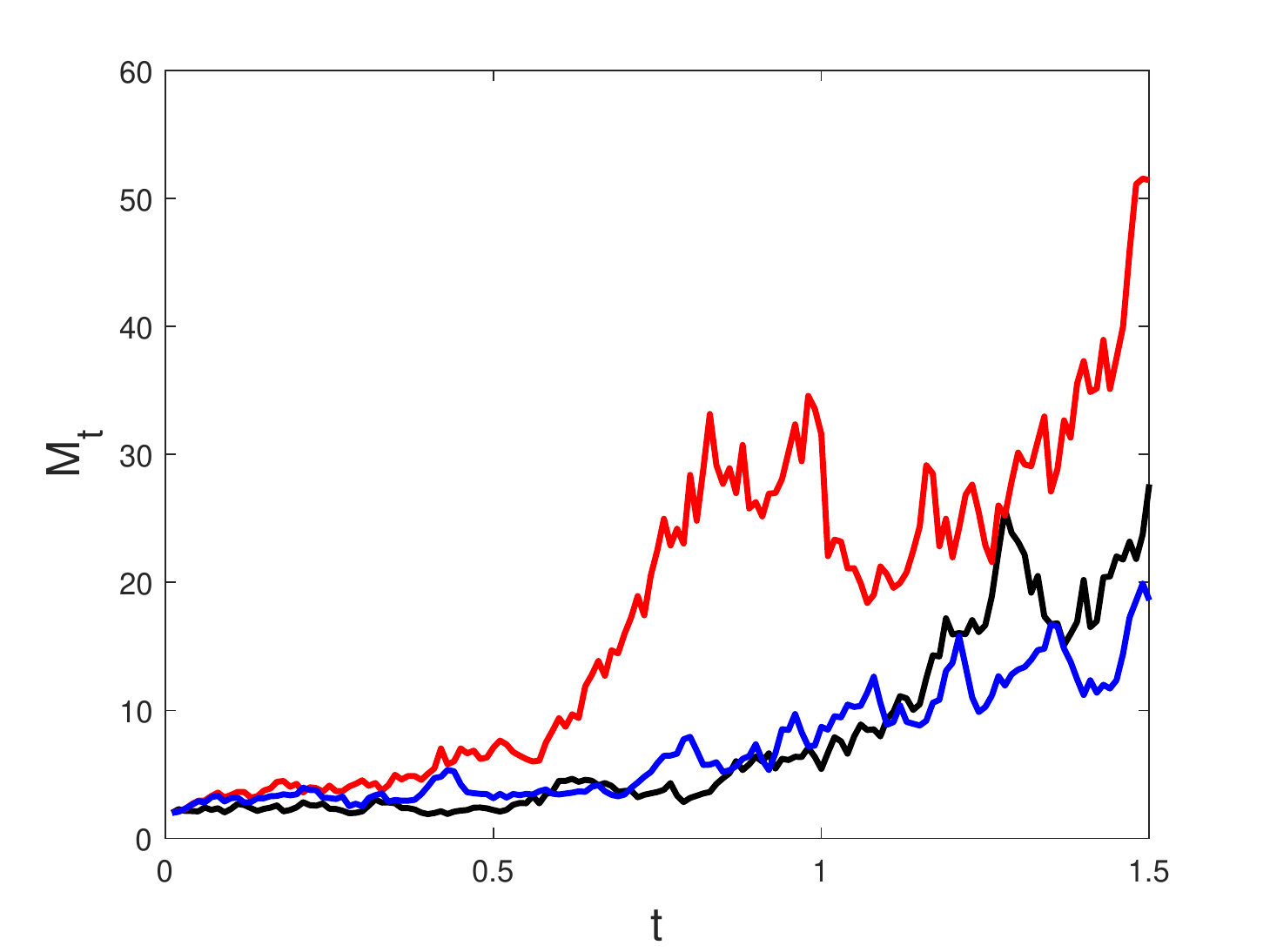}
    }\hspace{-8mm}
  \subfigure[]{
        \includegraphics[width=7.6cm]{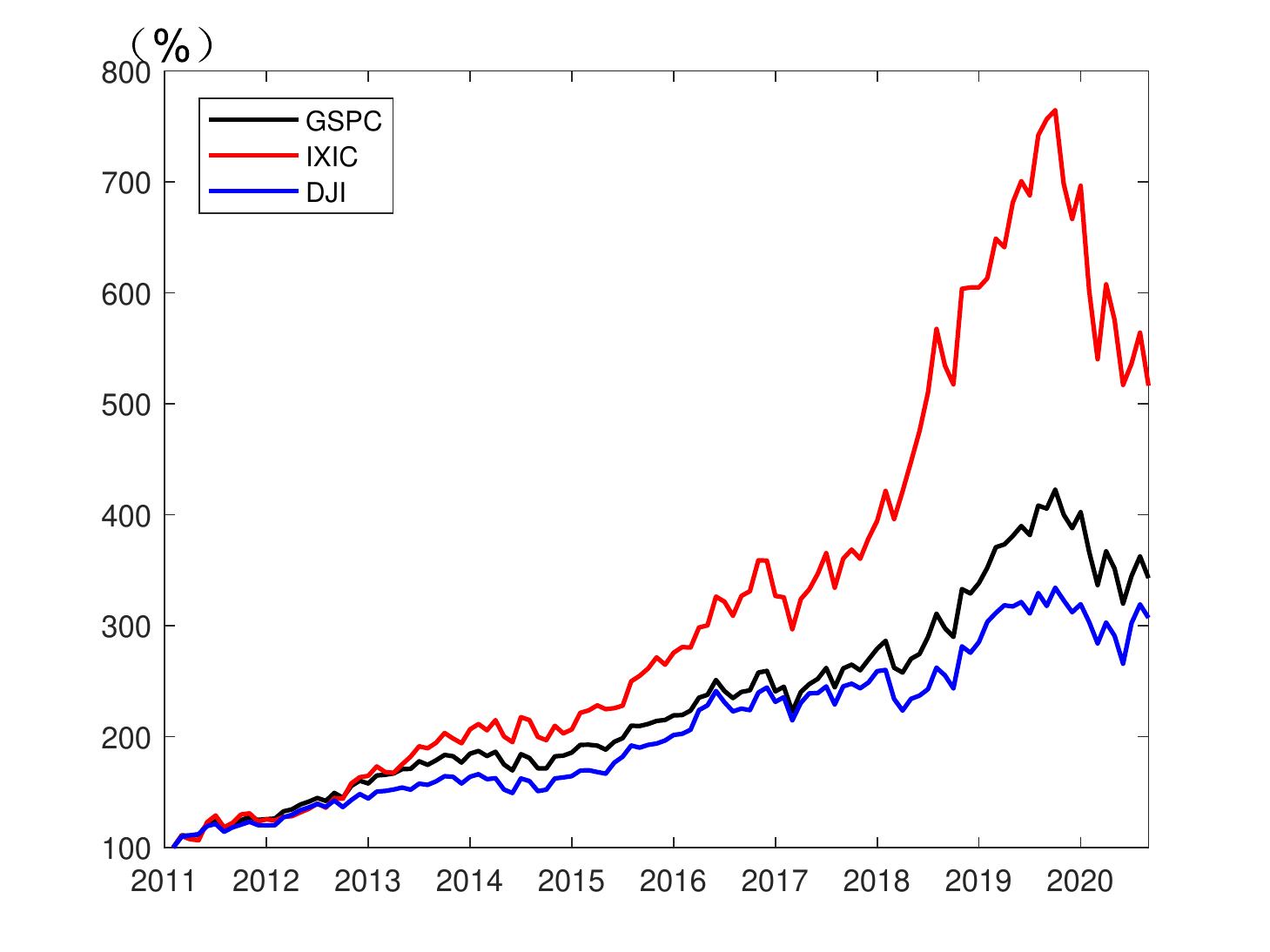}
    }

 \subfigure[]{
        \includegraphics[width=7.1cm]{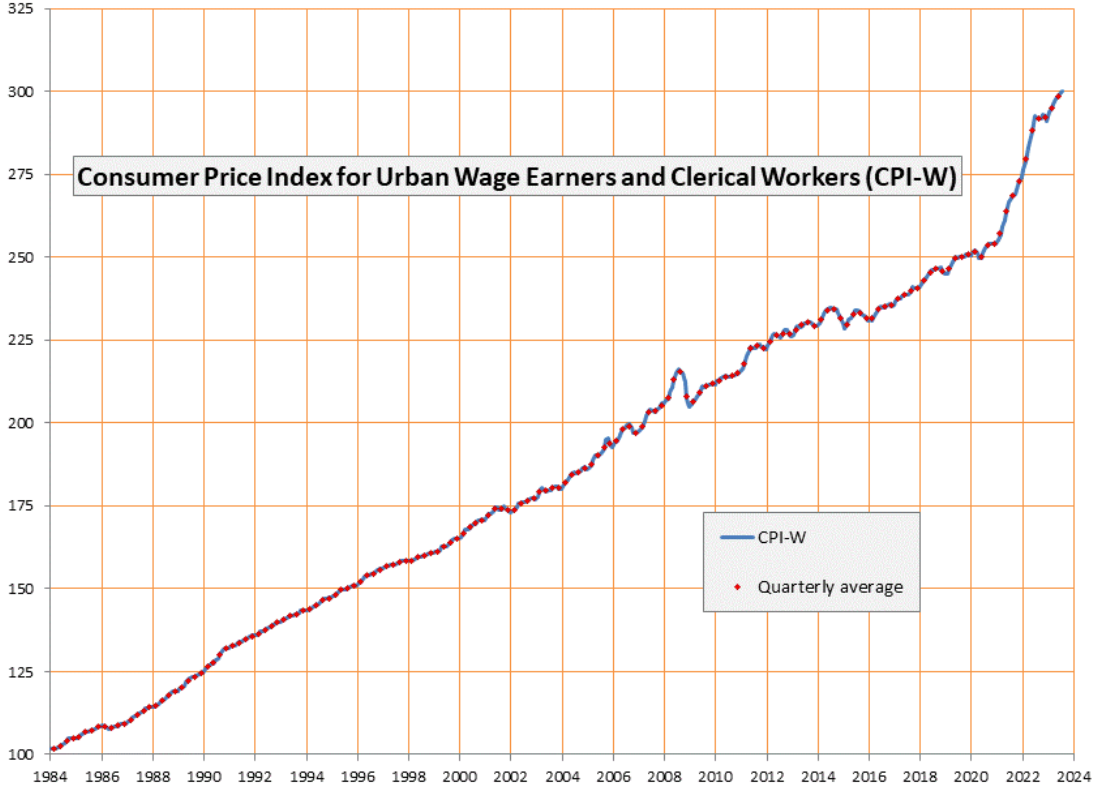}
    }
 \subfigure[]{
        \includegraphics[width=7.1cm,height=5cm]{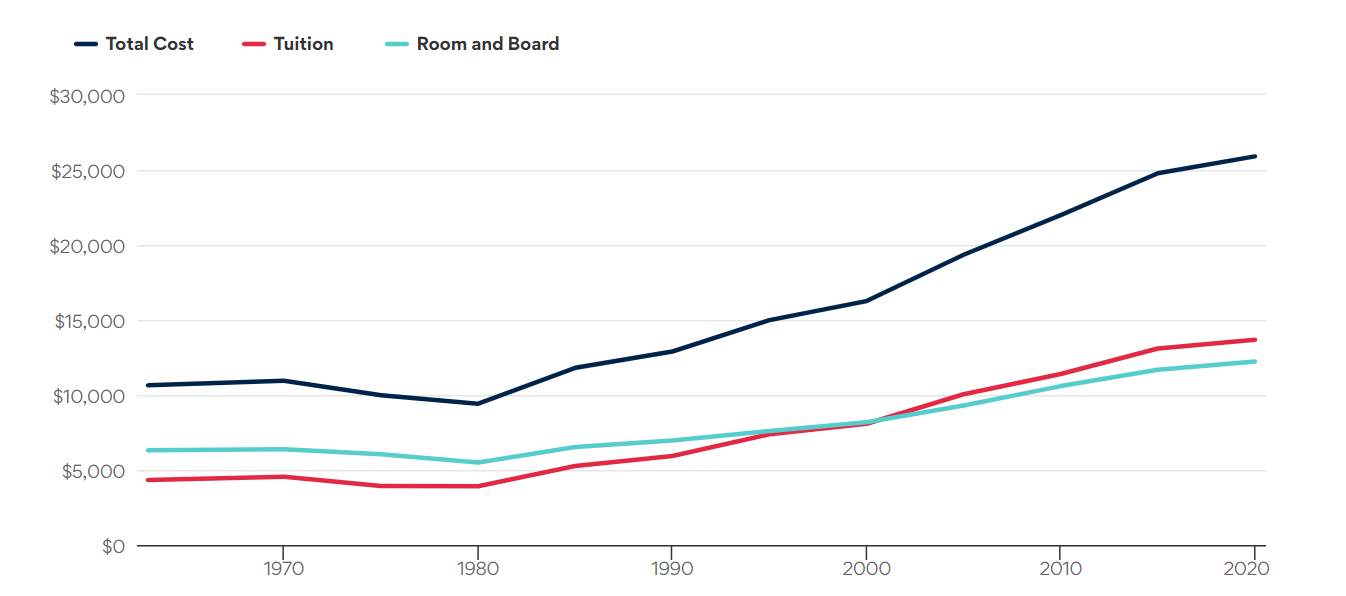}
    }
 \caption{\small (a): Simulated sample paths of the benchmark process $t\to M_t$ via Monte Carlo with dimension $d=1$. The model parameters are set to be $z=0.8,~{\rm m}=0,~b=1,~\mu_Z=2,~\sigma_Z=1,~\mu_B=2,~\sigma_B=0,~\gamma=\eta=1$. (b): The price movements of market indices S\&P500 (GSPC), NASDAQ (IXIC) and Dow Jones (DJI) based on observed data (April 1, 2011 to November 01, 2020) from Yahoo Finance. (c): Consumer Price Index for the US's Urban Wage Earners and Clerical Workers from 1984 to 2023, available from \url{https://www.ssa.gov/oact/STATS/cpiw_graph.html}. (d) Total cost of college in the U.S. for undergraduate students from 1963 to 2021, available from \url{https://www.bestcolleges.com/research/college-costs-over-time/}.}\label{fig:index}
\end{figure}

We consider the relaxed benchmark tracking using the capital injection. At any time $t\geq0$, it is assumed that the fund manager can strategically inject capital $A_t$ such that the total wealth $V_t+A_t$ stays above the benchmark process $M_t$. In the objective function, in addition to the expected utility on consumption, the fund manager also needs to take into account the cost of total capital injection. Mathematically speaking, the fund manager now aims to maximize the following objective function under dynamic floor constraint that, for all $(\mathrm{v},\textrm{m},z,b)\in\R_+\times\R_+\times\R_+\times\R$ with $\R_+:=[0,\infty)$,
\begin{align}\label{eq_prob_IBP}
\begin{cases}
\displaystyle {\rm w}(\mathrm{v},\textrm{m},z,b):=\sup_{(\theta,c,A)\in\mathbb{U}} \Ex\left[\int_0^{\infty} e^{-\rho t} U(c_t)dt - \beta\left(A_0+\int_0^{\infty} e^{-\rho t}dA_t\right) \right],\\[1.4em] \displaystyle ~\textrm{s.t.}~ M_t \le A_t + V^{\theta,c}_t~\textrm{at each}~t\geq0,
\end{cases}
\end{align}
where $\rho>0$ is the discount rate and $\beta>0$ describes the cost per injected capital, which can also be interpreted as the weight of relative importance between the consumption performance and the cost of capital injection. Here, $(\theta,c,A)\in\mathbb{U}$ denotes an admissible control where $(\theta,c)=(\theta_t,c_t)_{t\geq 0}$ is an $\Fx$-adapted process taking values on $\R^d\times\R_+$, and $A=(A_t)_{t\geq 0}$ is a right-continuous, non-decreasing and $\Fx$-adapted process. In the present paper, the utility function is considered as the power utility $U(x)=\frac{1}{p}x^{p}$, $x\in\R_+$, with the risk aversion parameter $p\in(-\infty,0)\cup(0,1)$.

Stochastic control problems with minimum guaranteed floor constraints have been studied in different contexts, see among \cite{Karoui05}, \cite{KM06}, \cite{Giacinto11}, \cite{Sekine12}, \cite{Giacinto14} and \cite{CYZ20} and references therein. In previous studies, the minimum guaranteed level is usually chosen as constant or deterministic level and some typical techniques to handle the floor constraints are to introduce the option based portfolio or the insured portfolio allocation such that the floor constraints can be guaranteed. However, if we consider the Merton problem under the strict floor constraint on wealth that $V^{\theta,c}_t\geq M_t$ a.s. for all $t\geq 0$, the set of admissible controls might be empty due to the more complicated benchmark process $M_t$ in \eqref{eq:Mtn}. In this regard, we introduce the singular control of capital injection $A_t$ in our relaxed tracking formulation such that the admissible set can be enlarged and the optimal control problem can become solvable. By minimizing the cost of capital injection, the controlled wealth process $V^{\theta,c}_t$ stays very close to the benchmark process $M_t$ as desired. To address the dynamic floor constraint, our first step is to reformulate it into an unconstrained control problem. By applying Lemma 2.4 in \cite{BoLiaoYu21}, for each fixed regular control $(\theta,c)$, the optimal singular control $A^{(\theta,c),*}_t$ satisfies the form that
\begin{align}\label{A-sing}
A^{(\theta,c),*}_t =0\vee \sup_{s\leq t}(M_{s}-V_{s}^{\theta,c}), \quad\forall t\geq 0.
\end{align}
Thus, the original problem \eqref{eq_prob_IBP} with the constraint $M_t\leq A_t+V_t^{\theta,c}$ for all $t\geq 0$ admits an equivalent formulation as an unconstrained utility maximization problem with a running maximum cost that
\begin{align}\label{eq_orig_pb}
{\rm w}(\mathrm{v}, \textrm{m},z,b)&=-\beta (\textrm{m}\vee b+z-\mathrm{v})^+\\
 &\quad+ \sup_{(\theta,c)\in\mathbb{U}^{\rm r}}\ \Ex\left[ \int_0^{\infty} e^{-\rho t} U(c_t)dt-\beta\int_0^{\infty} e^{-\rho t} d\left(0\vee \sup_{s\leq t}(M_{s}-V_{s}^{\theta,c})\right)\right].\nonumber
\end{align}
Here, $\mathbb{U}^{\rm r}$  denotes the set of regular $\Fx$-adapted admissible strategies $(\theta,c)=(\theta_t,c_t)_{t\geq 0}$ taking values on $\R^d\times\R_+$ such that, for any $T>0$, the SDE \eqref{eq:wealth2} admits a weak solution on $[0,T]$.

It is worth noting that some existing studies can be found in stochastic control problems with a running maximum cost, see \cite{BaIshii89},  \cite{BDR1994}, \cite{BPZ15}, \cite{Weerasinghe16} and \cite{Kroner18}, where the viscosity solution approach usually plays the key role. In our optimal control problem \eqref{eq_orig_pb}, two fundamental questions need to be addressed: (i) Can we characterize the optimal portfolio and consumption control pair $(\theta^*, c^*)$ in the feedback form if it exists? (ii) Whether the relaxed tracking formulation is well-defined in the sense that the expected total capital injection $\mathbb{E}[\int_0^{\infty} e^{-\rho t} d(0\vee \sup_{s\leq t}(M_{s}-V_{s}^{\theta^*,c^*}))]$ is finite? We will verify that our problem formulation does not require the injection of infinitely large capital to meet the tracking goal. The present paper contributes positive answers to both questions.

In solving the stochastic control problem \eqref{eq_orig_pb} with a running maximum cost, we introduce two auxiliary state processes with reflections and study an auxiliary stochastic control problem, which gives rise to the HJB equation with two Neumann boundary conditions. By applying the dual transform and stochastic flow analysis, we can conjecture and carefully verify that the classical solution of the dual PDE satisfies a separation form of three terms, all of which admit probabilistic representations involving some dual reflected diffusion processes and/or the local time at the reflection boundary. We stress that the main challenge is to prove the smoothness of the conditional expectation of the integration of an exponential-like functional of the reflected drifted-Brownian motion (RDBM) with respect to the local time of another correlated RDBM. We propose a new method of decomposition-homogenization to the dual PDE, which allows us to show the smoothness of the conditional expectation of the integration of exponential-like functional of the RDBM with respect to the local time of an independent RDBM.

By using the classical solution to the dual PDE with Neumann boundary conditions and establishing some technical estimations of candidate optimal controls, we can address the previous question (i) and rigorously characterize the optimal control pair $(\theta^*, c^*)$ in a feedback form in the verification theorem. Based on our estimations of the optimal control processes, we can further answer the previous question; (ii) and verify that the expected total capital injection $\mathbb{E}[\int_0^{\infty} e^{-\rho t} d(0\vee \sup_{s\leq t}(M_{s}-V_{s}^{\theta^*,c^*}))]$ is indeed bounded, and hence our problem \eqref{eq_prob_IBP} in a relaxed tracking formulation using the additional singular control is well defined. Moreover, it is also shown that $\mathbb{E}[\int_0^{\infty} e^{-\rho t} d(0\vee \sup_{s\leq t}(M_{s}-V_{s}^{\theta^*,c^*}))]$ is bounded below by a positive constant, indicating that the capital injection is necessary for the well-posedness for the control problem. We also note that $A^{*}_t =\sup_{s\leq t}(V_{s}^{\theta^*,c^*}-M_s)^{-}$ records the largest shortfall when the wealth process $V_{s}^{\theta^*,c^*}$ falls below the benchmark process $m_s$ up to time $t$. As a manner of risk management, the finite expectation $\mathbb{E}[\int_0^{\infty} e^{-\rho t} d(0\vee \sup_{s\leq t}(M_{s}-V_{s}^{\theta^*,c^*}))]$ can  quantitatively reflect the expected largest shortfall of the wealth management with respect to the benchmark in a long run.

The rest of the paper is organized as follows. In Section \ref{auxiliaryproblem}, we introduce the auxiliary state processes with reflections and derive the associated HJB equation with two Neumann boundary conditions for the auxiliary stochastic control problem. In Section \ref{sec:dualhjb}, we address the solvability of the dual PDE problem by verifying a separation form of the solution and the probabilistic representations, the homogenization of Neumann boundary conditions and the stochastic flow analysis. The verification theorem on the optimal feedback control is presented in Section \ref{sec:verfication} together with the technical proofs on the strength of stochastic flow analysis and estimations of the optimal control. It is also verified therein that the expected total capital injection is bounded. Finally, the proof of an auxiliary lemma is reported in Appendix \ref{appendix}.

\hypertarget{controlproblem}{
\section{Formulation of the Auxiliary Control Problem}\label{auxiliaryproblem}}

In this section, we formulate and study a more tractable auxiliary stochastic control problem, which is mathematically equivalent to the unconstrained optimal control problem \eqref{eq_orig_pb}. To this end, we first introduce a new auxiliary state process to replace the wealth process $V^{\theta,c}=(V_t^{\theta,c})_{t\geq 0}$ given in \eqref{eq:wealth2}. Let us first define
\begin{align}\label{eq:Dt}
    D_t:=M_t-V_t^{\theta,c}+\mathrm{v}-\textrm{m}\vee b-z,\quad \forall t\geq 0,
\end{align}
where $M=(M_t)_{t\geq 0}$ is defined by \eqref{eq:Mtn}, and it is clear that $D_0=0$. Moreover, for any $x\geq 0$, we define the running maximum process of the process $D=(D_t)_{t\geq 0}$ that
\begin{align}\label{eq:maxM}
L_t :=x\vee \sup_{s\leq t}D_s-x\geq 0,\quad \forall t\geq 0
\end{align}
with the initial value $L_0=0$. The auxiliary state process $X=(X_t)_{t\geq 0}$ is then defined as the reflected process $X_t:=L_t-D_t$ for $t\geq 0$ that satisfies the SDE that for all $t>0$,
\begin{align}\label{state-X}
X_t =x+\int_0^t\theta_s^{\top}\mu ds+\int_0^t\theta_s^{\top}\sigma dW_s  -\int_0^t c_s ds-\int_0^t \mu_Z Z_sds-\int_0^t \sigma_Z Z_sdW^{\eta}_s-\int_0^t dm_s+ L_t
\end{align}
with the initial value $X_0=x\geq 0$. In particular, $X_t$ hits $0$ if  the running maximum process $L_t$ increases. We will change the notation from $L_t$ to $L_t^X$ from this point onwards to emphasize its dependence on the new state process $X$ given in \eqref{state-X}.

On the other hand, for the running maximum process $m=(m_t)_{t\geq 0}$ in \eqref{eq:Mtn}, we also introduce a second auxiliary state process $I_t:=m_t-B_t$ for all $t\geq 0$. As a result, $I_t$ hits $0$ if $m_t$ increases, and we have
\begin{align}\label{state-Y}
I_t=I_0+\int_0^t dm_s-\int_0^t \mu_Bds-\int_0^t\sigma_B dW^{\gamma}_s,\quad \forall t\geq 0,
\end{align}
where the initial state value $I_0=\textrm{m}\vee b-b\geq0$.

The stochastic control problem \eqref{eq_orig_pb} can be solved by studying the auxiliary problem that, for all $(x,y,z)\in\R_+^3$,
\begin{align}\label{eq_mfg-2}
\begin{cases}
\displaystyle u(x,h,z):=\sup_{(\theta, c)\in\mathbb{U}^{\rm r}}J(x,h,z;\theta,c)=\sup_{(\theta, c)\in\mathbb{U}^{\rm r}} \Ex_{x,h,z}\left[\int_0^{\infty} e^{-\rho t} U(c_t)dt- \beta \int_0^{\infty} e^{-\rho t}dL_t^X\right],\\[1em]
\displaystyle~\text{S.t. the state process}~(X,I)~\text{satisfies~\eqref{state-X}-\eqref{state-Y}, and $Z$ is the GBM in \eqref{eq:Mtn},}
\end{cases}
\end{align}
where $\Ex_{x,h,z}[\cdot]:=\Ex[\cdot|X_0=x,I_0=h,Z_0=z]$. We note the equivalence that 
\begin{align}\label{eq:wu}
{\rm w}({\rm v},\textrm{m},z,b)=
\begin{cases}
u({\rm v}-\textrm{m}\vee b-z,\textrm{m}\vee b-b,z), &\text{if}~{\rm v}\geq \textrm{m}\vee b+z,\\[0.3em]
u(0,\textrm{m}\vee b-b,z)- \beta(\textrm{m}\vee b+z  -{\rm v}), &\text{if}~{\rm v}<\textrm{m}\vee b+z,
\end{cases}
\end{align}
where ${\rm w}({\rm v},\textrm{m},z,b)$ is given by \eqref{eq_orig_pb}.

We first have the following property of the value function $u$ in \eqref{eq_mfg-2}, whose proof is deferred to Appendix \ref{appendix}.
\begin{lemma}\label{lem:propoertyu}
Let the discount factor $\rho>2\mu_Z+\sigma_Z^2$ (if $2\mu_Z+\sigma_Z^2\leq0$, this condition is automatically satisfied). Then, $x\to u(x,h,z)$, $h\to u(x,h,z)$ and $z\to u(x,h,z)$ are non-decreasing. Moreover, it holds that, for all $(x_1,x_2,h_1,h_2,z_1,z_2)\in\R_+^6$,
\begin{align*}
\left|u(x_1,h_1,z_1)-u(x_2,h_2,z_2)\right|&\leq \beta (|x_1-x_2|+|h_1-h_2|)\nonumber\\
&\quad+\beta \left(\sigma_Z^2+\frac{|\mu_Z|}{\rho-\mu_Z}+\frac{3}{\rho-2\mu_Z-\sigma_Z^2}\right) |z_1-z_2|,
\end{align*}
where we recall that $\beta>0$ is the cost parameter due to the capital injection appeared in \eqref{eq_prob_IBP}.
\end{lemma}

By dynamic program argument, we can derive the associated HJB equation that, for $(x,h,z)\in(0,\infty)^3$,
\begin{align}\label{HJB}
\begin{cases}
\displaystyle \sup_{\theta\in\R^d}\left[\theta^{\top}\mu u_x+\frac{1}{2}\theta^{\top}\sigma\sigma^{\top}\theta u_{xx}+\sigma_Z \theta^{\top}\sigma\eta z(u_{xx}-u_{xz})-\sigma_B \theta^{\top}\sigma\gamma u_{xh}\right]+\sup_{c\geq0}\left(\frac{c^p}{p}-cu_x\right)\\[0.8em]
\displaystyle~ +\frac{1}{2}\sigma_B^2u_{hh} -\mu_B u_h+\frac{1}{2}\sigma_Z^2z^2(u_{zz}+u_{xx}-2u_{xz})+\mu_Z z(u_z-u_x)\\[0.8em]
\displaystyle\qquad\qquad\qquad+\sigma_Z\sigma_B z\eta^{\top}\gamma (u_{xh}-u_{hz})=\rho u,\\[1em]
\displaystyle u_x(0,h,z)=\beta,\quad \forall (h,z)\in\R_+^2,\\[0.6em]
\displaystyle u_h(x,0,z)= u_x(x,0,z),\quad \forall (x,z)\in\R_+^2.
\end{cases}
\end{align}
Here, the first Neumann boundary condition in \eqref{HJB} stems from the fact that $X_t=0$ when $L_t$ increases; while the second Neumann boundary condition in \eqref{HJB} comes from the fact that $I_t=0$ when $m_t$ increases.

Assuming that HJB equation \eqref{HJB} admits a unique classical solution $u$ satisfying $u_{xx}<0$ and $u_x\geq0$, which will be verified in later sections, the first-order condition yields the candidate optimal feedback control that
\begin{align*}
\theta^*=-(\sigma\sigma^{\top})^{-1}\frac{\mu u_x-\sigma_B\sigma\gamma u_{xh}+\sigma_Z \sigma\eta z(u_{xx}-u_{xz})}{u_{xx}},\qquad c^*=(u_x)^{\frac{1}{p-1}}.
\end{align*}
Plugging the above results into \eqref{HJB}, we apply Lemma~\ref{lem:propoertyu} and the Legendre-Fenchel transform of the solution $u$ only with respect to $x$ that $\hat{u}(y,h,z):=\sup_{x\geq0}\{u(x,h,z)-yx\}$ for all $(y,h,z)\in(0,\beta]\times\R_+^2$. Equivalently, $u(x,h,z)=\inf_{y\in (0,\beta]}\{ \hat{u}(y,h,z)-xy\}$ for all $(x,h,z)\in\R_+^3$. The dual transform can linearize the HJB equation \eqref{HJB}, and we get the dual PDE for $\hat{u}(y,h,z)$ that, for all $(y,h,z)\in(0,\beta]\times(0,\infty)^2$,
\begin{align}\label{eq:dual-u}
\frac{\alpha^2}{2} y^2\hat{u}_{yy}+\rho y\hat{u}_y+\frac{\sigma_B^2}{2} \hat{u}_{hh}-\mu_B\hat{u}_h++\frac{1}{2}\sigma_Z^2z^2 \hat{u}_{zz}+\mu_Z z\hat{u}_z+\kappa_1 y\hat{u}_{yh}-\kappa_2  zy\hat{u}_{yz}\nonumber\\
+\sigma_Z\sigma_B\gamma^{\top}\eta z\hat{u}_{zh}+(\kappa_2 -\mu_Z)zy+\left(\frac{1-p}{p}\right)y^{-\frac{p}{1-p}}=\rho \hat{u},
\end{align}
where the coefficients
\begin{align}\label{eq:coeffalpha}
\alpha:=(\mu^{\top}(\sigma\sigma^{\top})^{-1}{\mu})^{\frac{1}{2}}>0,\quad \kappa_1:=\sigma_B\mu^{\top}(\sigma\sigma^{\top})^{-1}\sigma\gamma,\quad \kappa_2:=\sigma_Z\mu^{\top}(\sigma\sigma^{\top})^{-1}\sigma\eta.
\end{align}
Correspondingly, the first Neumann boundary condition in \eqref{HJB} is transformed to the Neumann boundary condition that 
\begin{align}\label{b-v-1}
\hat{u}_y(\beta, h,z)=0,\quad \forall (h,z)\in\R_+^2.
\end{align}
The second Neumann boundary condition in \eqref{HJB} is transformed to the Neumann boundary condition that
\begin{align}\label{b-v-2}
\hat{u}_z(y,0,z)=y,\quad \forall (y,z)\in (0,\beta]\times\R_+.
\end{align}

\hypertarget{sec:dualhjb}{%
\section{Solvability of the Dual PDE}\label{sec:dualhjb}}

This section examines the existence of solution to PDE~\eqref{eq:dual-u} with two Neumann boundary conditions \eqref{b-v-1} and \eqref{b-v-2} in the classical sense using the probabilistic approach.

Before stating the main result of this section, let us first introduce the following function that, for all $(r,h,z)\in\R_+^3$,
\begin{align}\label{eq:v}
v(r,h,z)&:=\frac{1-p}{p}\beta^{-\frac{p}{1-p}}\Ex\left[\int_0^{\infty} e^{-\rho s+\frac{p}{1-p}R_s^{r}}ds\right]+\beta(\kappa_2 -\mu_Z)\Ex\left[\int_0^{\infty} e^{-\rho s-R_s^{r}}N_s^zds\right]\nonumber\\
&\quad-\beta\Ex\left[ \int_0^{\infty}e^{-\rho s-R^{r}_s}dK_s^{h}\right].
\end{align}
Here, the processes $R^{r}=(R_t^{r})_{t\geq 0}$ and $H^{h}=(H_t^{h})_{t\geq 0}$ with $(r,h)\in\R_+^2$ are two reflected processes satisfying that, for all $t\geq 0$,
\begin{align}
R_t^{r}&=r+\left(\frac{\alpha^2}{2}-\rho\right)t+\alpha B^1_{t}+L_{t}^{r}\geq 0,\label{eq:R}\\
H_t^{h}&=h- \mu_B t-\varrho_1 \sigma_B B^1_{t}-\sqrt{1-\varrho_1^2} \sigma_B B^0_{t}+K_{t}^{h}\geq 0,\label{eq:H}
\end{align}
and the process $N^{z}=(N_t^{z})_{t\geq 0}$ is a GBM satisfying
\begin{align*}\label{eq:N}
dN_t^z=\mu_Z N_t^zdt+\varrho_2 \sigma_Z N_t^zdB^1_t+\sqrt{1-\varrho_2^2} \sigma_Z N_t^z B^2_{t},\quad N_0^z=z,
\end{align*}
where $B^0=(B^0_t)_{t\geq 0}$, $B^1=(B^1_t)_{t\geq 0}$ and $B^2=(B^2_t)_{t\geq 0}$ are three independent scalar Brownian motions; while $L^{r}=(L_t^{ r})_{t\geq 0}$ (resp. $K^{h}=(K_t^{h})_{t\geq 0}$) is a continuous and non-decreasing process that increases only on the set $\{t\geq 0;~R_t^{ r}=0\}$ with $L_0^{ r}=0$ (resp. $\{t\geq 0;~H_t^{ h}=0\}$ with $K_0^{h}=0$) such that $R_t^r\geq0$ (resp. $H_t^h\geq0$) a.s. for $t\geq0$, the correlative coefficients are respectively given by
\begin{align}
\varrho_1:=\frac{\mu^{\top}(\sigma\sigma^{\top})^{-1}\sigma\gamma}{\alpha},\quad \varrho_2:=\frac{\mu^{\top}(\sigma\sigma^{\top})^{-1}\sigma\eta}{\alpha}.
\end{align}

On the other hand, note that $(R_t^{ r})_{t \geq 0}$ in \eqref{eq:R} and $(H_t^{h})_{t \geq 0}$ in \eqref{eq:H} are RDBMs, and the processes $L^r=(L_t^r)_{t\geq0}$ and $K^h=(K_t^h)_{t\geq0}$ are uniquely determined by the above properties (c.f. \cite{Harrison85}). Using the solution representation of ``the Skorokhod problem", it follows that, for all $t\geq 0$,
\begin{align}\label{eq:L}
\begin{cases}
\displaystyle L_t^{r} =0 \vee\left\{-r+\max_{s\in[0,t]}\left[-\alpha B_{s}^1-\left(\frac{\alpha^2}{2}-\rho\right)s\right]\right\},\\[1.2em]
\displaystyle K_t^{ h} =0 \vee\left\{-h+\max_{s\in[0,t]}\left(\mu_Bs+\sigma_BB^3_{s}\right)\right\},
\end{cases}
\end{align}
where the process $B^3=(B^3_t)_{t\geq 0}=(\varrho_1 B^1_t+\sqrt{1-\varrho_1^2}B^0_t)_{t\geq 0}$ is a scalar Brownian motion.

The main result of this section is stated as follows:
\begin{theorem}\label{thm:dualPDEclass}
Assume $\rho>\frac{\alpha^2 p}{2(1-p)}$ and $\mu_Z> \kappa_2$. Consider the function $v(r,h,z)$ for $(r,h,z)\in\R_+^3$ defined by the probabilistic representation \eqref{eq:v}. For all $(y,h,z)\in(0,\beta]\times\R_+^2$, let us define
\begin{align}\label{eq:hatu}
\hat{u}(y,h,z):= v\left(-\ln \frac{y}{\beta},h,z\right).
\end{align}
Then, for each $(h,z)\in \R_+^2$, we have
\begin{itemize}
\item the function $(0,\beta]\ni y \mapsto \hat{u}(y,h,z)$ is strictly convex.
\item the function $\hat{u}(y,h,z)$ is a classical solution of PDE~\eqref{eq:dual-u} with Neumann boundary conditions \eqref{b-v-1} and \eqref{b-v-2}.
\end{itemize}
On the other hand, if the Neumann problem \eqref{eq:dual-u}-\eqref{b-v-2} has a classical solution $\hat{u}(y,h,z)$ satisfying $|\hat{u}(y,h,z)| \leq C(1+|y|^{-q}+z^q)$ for some $q>1$ and some constant constant $C>0$, then $v(r,h,z):= \hat{u}(\beta e^{-r},h,z)$ admits the probabilistic representation \eqref{eq:v}.
\end{theorem}

Theorem~\ref{thm:dualPDEclass} provides a probabilistic presentation of the classical solution to the PDE \eqref{eq:dual-u} with Neumann boundary conditions \eqref{b-v-1} and \eqref{b-v-2}. Our method in the proof of Theorem~\ref{thm:dualPDEclass} is completely from a probabilistic perspective. More precisely, we start with the proof of the smoothness of the function $v$ by applying properties of reflected processes $(R^{r},H^{h},N^z)$, the homogenization technique of the Neumann problem and the stochastic flow analysis. Then, we show that $v$ solves a linear PDE by verifying two related Neumann boundary conditions at $r=0$ and $h=0$ respectively.

The next result deals with the first two terms of the function $v$ given in \eqref{eq:v}, whose proof is similar to that of Theorem~4.2 in \cite{BoLiaoYu21} after minor modifications. For the completeness, we provide a sketch of the proof in Appendix~\ref{appendix}.
\begin{lemma}\label{lem:v1stpart}
 Assume $\rho>\frac{\alpha^2 |p|}{2(1-p)}+\mu_Z$ and $\mu_Z> \kappa_2$.  For any $(r,z)\in \R_+^2$, denote by $l(r,z)$ the sum of the first term  and the second term of the function $v$ given in \eqref{eq:v} that
\begin{align}\label{eq:fcnl}
l(r,z):=\frac{1-p}{p}\beta^{-\frac{p}{1-p}}\Ex\left[\int_0^{\infty} e^{-\rho s+\frac{p}{1-p}R_s^{r}}ds\right]+\beta(\kappa_2 -\mu_Z)\Ex\left[\int_0^{\infty} e^{-\rho s-R_s^{r}}N_s^zds\right].
\end{align}
Then, the function $l(r,z)$ is a classical solution to the following Neumann problem with Neumann boundary condition at $r=0$:
\begin{align}\label{eq:HJB-l}
\begin{cases}
\displaystyle \frac{\alpha^2}{2} l_{rr}+\left(\frac{\alpha^2}{2}-\rho\right) l_r+\frac{1}{2}\sigma_Z^2 z^2 l_{zz}+\mu_Z z l_z+\kappa_2 z l_{rz}\\[0.9em]
\displaystyle\qquad\qquad\quad+(\kappa_2-\mu_Z)\beta z e^{-r}+\frac{1-p}{p}\beta^{-\frac{p}{1-p}}e^{\frac{p}{1-p}r}=\rho l,~\text{on}~(0,\infty)^2,\\[0.9em]
\displaystyle l_r(0,z)=0,\quad \forall z\in\R_+.
\end{cases}
\end{align}
On the other hand,  if the Neumann problem \eqref{eq:HJB-l} has a classical solution $l(r,z)$ for $r\in\R_+$ satisfying $|l(r,z)|\leq C(1+e^{qr}+z^q)$ for some $q>1$ and a constant $C>0$ depending on $(\mu,\sigma,\mu_Z,\sigma_Z,p)$, then this solution $l(r,z)$ admits the probabilistic representation \eqref{eq:fcnl}. Moreover, $l(r,z)$ admits the explicit form that
\begin{align}\label{eq:explicit-l}
l(r,z)=C_1 \beta^{-\frac{p}{1-p}}e^{\frac{p}{1-p}r}+C_2 \beta e^{-r}+z\left(\beta e^{-r}-\frac{\beta}{\ell}e^{-\ell r}\right),\quad \forall (r,z)\in\R_+^2,
\end{align}
where $C_1,C_2>0$ are two positive constants defined by
\begin{align}\label{eq:C1C2}
C_1:=\frac{2(1-p)^3}{p(2\rho(1-p)-\alpha^2 p)},\quad C_2:=\frac{2(1-p)^2}{2\rho(1-p)-\alpha^2 p}\beta^{-\frac{1}{1-p}},
\end{align}
the constant $\ell$ is the positive root of the quadratic equation
\begin{align}\label{kappaeq}
\frac{1}{2}\alpha^2 \ell^2+\left(\rho-\kappa_2-\frac{1}{2}\alpha^2\right)\ell+\mu_Z-\rho=0,
\end{align}
which is given by
\begin{align}\label{eq:kappa}
    \ell=\frac{-(\rho-\kappa_2-\frac{1}{2}\alpha^2)+\sqrt{(\rho-\kappa_2-\frac{1}{2}\alpha^2)^2+2\alpha^2(\rho-\mu_Z)}}{\alpha^2}>0.
\end{align}
\end{lemma}

The challenging step in our problem is to handle the last term of the function $v$ given in \eqref{eq:v}, which differs substantially from the first two terms of $v$ as it now involves both the reflected process $R^{r}$ and the local time term $K^{h}$ of the reflected process $H^{h}$. In particular, we highlight that the reflected process $R^r$ is not independent of the local time process $K^h$. As a preparation step to handle the smoothness of the second term of the function $v$ given in \eqref{eq:v}, let us first discuss the case when the reflected process $R^r$ is independent of the local time $K^{h}$.

\begin{lemma}\label{lem:secondtermm}
Let us consider the function that, for all $(r,h)\in \R_+^2$,
\begin{align}\label{eq:varphiindependent}
\varphi(r,h):=-\beta\Ex\left[\int_0^{\infty}e^{-\rho s-R^{r}_s}dG_s^{h}\right],
\end{align}
where the reflected process $R^r=(R_t^r)_{t\geq 0}$ with $r\in\R_+$ is given by \eqref{eq:R}, and the process $(P^h,G^h)=(P_t^h,G_t^h)_{t\geq 0}$ satisfies the reflected SDE:
\begin{align}\label{eq:hatH}
P_t^{h}&=h-\int_0^t \mu_Bds-\int_0^t \sigma_B dB^0_{s}+\int_0^t dG_{s}^{h}\geq 0.
\end{align}
Here, $G^{h}=(G_t^{h})_{t\geq 0}$ is a continuous and non-decreasing process that increases only on the time set $\{t\in\R_+;~P_t^{h}=0\}$ with $G_0^{h}=0$ and such that $P_t^h\geq0$ a.s. for $t\geq0$. Then, the processes $G^{h}=(G_t^{h})_{t\geq 0}$ and $R^r=(R_t^r)_{t\geq 0}$ are independent. Moreover, the function $\varphi(r,h)$ is a classical solution to the following PDE with Neumann boundary conditions at $r=0$ and $h=0$:
\begin{align}\label{eq:HJB-m}
\begin{cases}
\displaystyle \frac{\alpha^2}{2}\varphi_{rr}+\left(\frac{\alpha^2}{2}-\rho\right)\varphi_r+\frac{\sigma_B^2}{2} \varphi_{hh}-\mu_B \varphi_h=\rho \varphi,~ \text{on}~(0,\infty)^2,\\[1em]
\displaystyle \varphi_r(0,h)=0,\quad \forall h\in\R_+,\\[0.8em]
\displaystyle \varphi_h(r,0)=\beta e^{-r},\quad \forall r\in\R_+.
\end{cases}
\end{align}
On the other hand, if the Neumann problem \eqref{eq:HJB-m} has a classical solution $\varphi(r,h)$ satisfying $|\varphi(r,h)|\leq C$ for some constant $C>0$ depending on $(\mu,\sigma,\mu_B,\sigma_B)$, then this solution $\varphi(r,h)$ satisfies the probabilistic representation \eqref{eq:fcnm}.
\end{lemma}
The following result guarantees the smoothness of the function $\varphi(r,h)$ for $(r,h)\in\R_+^2$ defined by \eqref{eq:varphiindependent}.

\begin{lemma}\label{lem:derivative-m}
Consider the function $\varphi(r,h)$ for $(r,h)\in\R_+^2$ defined by \eqref{eq:varphiindependent}. Assume $\sigma_B\neq 0$, then it holds that $\varphi\in C^{2,2}(\R_+^2)$. Moreover, for all $(r,h)\in\R_+^2$, we have
\begin{align}
\varphi_r( r, h)
&=\beta\Ex\left[\int_0^{\tau_{r} } e^{-\rho s-R_s^{r}} d G_s^{h}\right],\quad \varphi_{h}(r,h)=\beta \Ex\left[e^{-\rho \eta_{h}-R^{r}_{ \eta_{h}}}\right],\label{eq:m-r}\\
\varphi_{rr}(r,h)&= \beta \int_{0}^{\infty}\int_{-\infty}^{r} e^{-\rho s-r+x}\phi_1(s,x,r)dx d\Ex[G_s^{h}]- \beta \Ex\left[\int_{0}^{\tau_{r}} e^{-\rho s-R_s^{r}}dG_s^{h}\right],\label{eq:m-rr}\\
\varphi_{rh}(r,h)&=-\beta \Ex\left[e^{-\rho \eta_{h}-R_{\eta_{h}}^{r}}\mathbf{1}_{\eta_h<\tau_r}\right]\!\!=\!\!-\beta\int_0^{\infty}\int_0^r\int_{-\infty}^y e^{-\rho s-r+x}\phi_1(s,x,y)\phi_2(s,h)dxdyds.\label{eq:m-rh}
\end{align}
Here $\tau_{r}:=\inf\{s \geq 0;~-\alpha B^1_s-(\frac{1}{2}\alpha^2-\rho)s=r\}$, $\eta_{h}:=\inf\{s \geq 0;~\sigma_B B^0_s+\mu_B s=h\}$ (with $\inf \emptyset=+\infty$ by convention), and functions $\phi_1(s,x,y),\phi_2(s,h)$ are respectively given by, for all $(s,x,y,h)\in \R_+^4$,
\begin{align}\label{eq:d}
\phi_1(s,x,y)&=\frac{2(2 y-x)}{\sqrt{2\hat{\sigma}^2 \pi s^3}}\exp\left(\frac{\hat{\mu}}{\hat{\sigma}} x-\frac{1}{2} \hat{\mu}^2 s-\frac{(2 y-x)^2}{2\hat{\sigma}^2 s}\right),\\
\phi_2(s,h)&=\frac{h}{\sqrt{2 \sigma_B^2\pi s^3}} \exp \left(-\frac{(h-\mu_B s)^2}{2 \sigma_B^2 s}\right),
\end{align}
where parameters $\hat{\mu}:=\frac{\alpha}{2}-\frac{\rho}{\alpha}$ and $\hat{\sigma}:=\alpha$.
\end{lemma}

\begin{proof}
The independence between the reflected process $R^r=(R_t^r)_{t\geq0}$ and the process $G^h=(G_t^h)_{t\geq0}$ plays an important role in the proof below. Before calculating the partial derivatives of $\varphi(r,h)$, we first claim that, for all $(r,h)\in \R_+^2$,
\begin{align}\label{eq:m-ex}
\varphi(r,h)=-\beta\int_0^{\infty}\Ex\left[e^{-\rho s-R^{r}_s}\right]d\Ex[G_s^{h}].
\end{align}
In fact, fix $(T,r,h)\in\R_+^3$, and let $n\geq 1$, $s_i=\frac{T}{n}i$ with $i=0,1,\ldots,n$. Then, it holds that
\begin{align*}
\Ex\left[\int_0^Te^{-\rho s-R^{r}_s}dG_s^{h}\right]=\Ex\left[\lim_{n\to \infty} \sum_{i=1}^n e^{-\rho s_{i}-R^{r}_{s_i}}\left(G_{s_i}^{h}-G_{s_{i-1}}^{h}\right)\right].
\end{align*}
Note that $e^{-\rho s-R_s^{r}}\leq 1$ a.s., for all $s\in[0,T]$. Then $\sum_{i=1}^n e^{-\rho s_{i}-R^{r}_{s_i}}(G_{s_i}^{h}-G_{s_{i-1}}^{h})\leq \sum_{i=1}^n(G_{s_i}^{h}-G_{s_{i-1}}^{h})=G_T^h$. Thus, from the dominated convergence theorem (DCT) and the independence between $B^0$ and $B^1$, it follows that
\begin{align}\label{eq:m-ex-T}
&\Ex\left[\int_0^T\!\!\!e^{-\rho s-R^{r}_s}dG_s^{h}\right]\!\!=\!\!\lim_{n\to \infty}\Ex\left[\sum_{i=1}^n e^{-\rho s_{i}-R^{r}_{s_i}}\left(G_{s_i}^{h}-G_{s_{i-1}}^{h}\right)\right]\!\!=\!\!\lim_{n\to \infty}\sum_{i=1}^n \Ex\left[e^{-\rho s_{i}-R^{r}_{s_i}}\left(G_{s_i}^{h}-G_{s_{i-1}}^{h}\right)\right]\nonumber\\
&\qquad\qquad=\lim_{n\to \infty}\sum_{i=1}^n \Ex\left[e^{-\rho s_{i}-R^{r}_{s_i}}\right]\left\{\Ex[G_{s_i}^{h}]-\Ex[G_{s_{i-1}}^{h}]\right\}=\int_0^T\Ex\left[e^{-\rho s-R^{r}_s}\right]d\Ex[G_s^{h}].
\end{align}
Letting $T\to \infty$ on both side of \eqref{eq:m-ex-T} and applying MCT, it follows that
\begin{align*}
\Ex\left[\int_0^{\infty} e^{-\rho s-R^{r}_s}dG_s^{h}\right]=\int_0^{\infty}\Ex\left[e^{-\rho s-R^{r}_s}\right]d\Ex[G_s^{h}].
\end{align*}
This verifies the claim \eqref{eq:m-ex}.

Let $h\in\R_+$ be fixed. First of all, we consider the case with arbitrary $r_2>r_1 \geq 0$. It follows from \eqref{eq:varphiindependent} that
\begin{align*}
\frac{\varphi(r_2,h)-\varphi(r_1,h)}{r_2-r_1}=-\beta\int_0^{\infty} \mathbb{E}\left[e^{-\rho s} \frac{e^{-R_s^{r_2}}-e^{-R_s^{ r_1}}}{r_2-r_1}d G_s^{h}\right] .
\end{align*}
A direct calculation yields that, for all $s \geq 0$,
\begin{align*}
&\lim _{r_2 \downarrow r_1} \frac{e^{-R_s^{r_2}-\rho s}-e^{-R_s^{ r_1}-\rho s}}{r_2-r_1}=\left\{\begin{array}{cc}
\displaystyle -e^{-R_s^{ r_1}-\rho s}, & \max_{q\in[0,s]}\{-\alpha B^1_q-(\frac{1}{2}\alpha^2-\rho)q\} \leq r_1, \\[0.6em]
\displaystyle 0, & \max_{q\in[0, s]}\{-\alpha B^1_q-(\frac{1}{2}\alpha^2-\rho)q\}>r_1 .
\end{array}\right.
\end{align*}
Note that $\sup_{(r_1,r_2)\in \R_+^2}|\frac{e^{-R_s^{r_2}-\rho s}-e^{-R_s^{r_1}-\rho s}}{r_2-r_1}|\leq e^{-\rho s}$. Then, the DCT yields that
\begin{align}\label{eq:m-r-1}
\lim _{r_2 \downarrow r_1}\frac{\varphi(r_2,h)-\varphi(r_1,h)}{r_2-r_1} &=\beta\mathbb{E}\left[ \int_0^{\infty} e^{-\rho s-R_s^{ r_1}}{\bf1}_{\{\max _{q\in[0, s]}\{-\alpha B^1_q-(\frac{1}{2}\alpha^2-\rho)q\} \leq r_1\}} dG_s^{h}\right] \nonumber\\
&=\beta\mathbb{E}\left[\int_0^{\tau_{r_1} } e^{-\rho s-R_s^{ r_1}} d G_s^{h}\right].
\end{align}
For the case $r_1>r_2 \geq 0$, similar to the computations for \eqref{eq:m-r-1}, we can show that
\begin{align*}
\lim _{r_2 \uparrow r_1} \frac{\varphi(r_2,h)-\varphi(r_1,h)}{r_2-r_1}=\lim _{r_2 \downarrow r_1} \frac{\varphi(r_2,h)-\varphi(r_1,h)}{r_2-r_1}.
\end{align*}
Thus, the representation \eqref{eq:m-r} holds.

For any real numbers $r_0, r_n \geq 0$ with $r_n \rightarrow r_0$ as $n \to\infty$, we have from \eqref{eq:m-r} that, for all $n \geq 1$,
\begin{align}\label{eq:Deltan}
\Delta_n&:=\frac{\varphi_r(r_n,h)-\varphi_r(r_0,h)}{r_n-r_0}\nonumber\\
&= \beta \Ex\left[\frac{1}{r_n-r_0}\int_{\tau_{r_0}}^{\tau_{r_n}} e^{-\rho s-R_s^{r_0}}dG_s^{h}\right]+ \beta \Ex\left[\frac{1}{r_n-r_0}\int_{0}^{\tau_{r_0}} e^{-\rho s}\left(e^{-R_s^{r_n}}-e^{-R_s^{r_0}}\right)dG_s^{h}\right]\nonumber\\
&\quad+\beta \Ex\left[\frac{1}{r_n-r_0}\int_{\tau_{r_0}}^{\tau_{r_n}} e^{-\rho s}\left(e^{-R_s^{r_n}}-e^{-R_s^{r_0}}\right)dG_s^{h}\right]:=\Delta_n^{(1)}+\Delta_n^{(2)}+\Delta_n^{(3)}.
\end{align}
In order to handle the term $\Delta_n^{(1)}$, we first focus on the case with $r_n \downarrow r_0$ as $n \to \infty$. Let us define the drifted-Brownian motion $\tilde{W}_s:=-\alpha B^1_s-(\frac{\alpha^2}{2}-\rho)s$ for all $s\in\R_+$. In view of \eqref{eq:R} and \eqref{eq:L}, it holds that
\begin{align*}
R_s^{r_0}=r_0-\tilde{W}_s+\left(\sup_{q\in[0,s]}\tilde{W}_q-r_0\right)^+,\quad \forall s\geq 0.
\end{align*}
Note that $\phi_1(s,x,y)$ defined by \eqref{eq:d} is the joint probability density of two-dimensional random variable $(\tilde{W}_s,\max_{q\in[0,s]}\tilde{W}_q)$ for any $s\geq 0$.  Then, we have
\begin{align}\label{eq:delat-n1-1}
&\Ex\left[\frac{1}{r_n-r_0}\int_{\tau_{r_0}}^{\tau_{r_n}} e^{-\rho s-R_s^{r_0}}dG_s^{h}\right]=\Ex\left[\frac{1}{r_n-r_0}\int_{0}^{\infty} e^{-\rho s-R_s^{r_0}}\mathbf{1}_{ \{\tau_{r_0}<s\leq \tau_{r_n}\}}dG_s^{h}\right]\nonumber\\
&\quad=\int_{0}^{\infty}\Ex\left[\frac{\exp(-\rho s-r_0+\tilde{W}_s-(\sup_{q\in[0,s]}\tilde{W}_q-r_0)^+)}{r_n-r_0} \mathbf{1}_{\{r_0<\max _{q \in[0, s]}\tilde{W}_q \leq r_n\}}\right]d\Ex[G_s^{h}]\nonumber\\
&\quad=\int_{0}^{\infty}\int_{r_0}^{r_n}\int_{-\infty}^y \frac{e^{-\rho s-r_0+x-(y-r_0)^+}}{r_n-r_0}\phi_1(s,x,y)dxdy d\Ex[G_s^{h}].
\end{align}
For $(s,y)\in\R_+^2$, set $g(s,y):=\int_{-\infty}^y e^{-\rho s-r_0+x-(y-r_0)^+}\phi_1(s,x,y)dx$.  Then, by the continuity of $y\to g(s,y)$,  we have
\begin{align}\label{eq:delat-n1-2}
\lim_{n\to \infty} \frac{1}{r_n-r_0}\int_{r_0}^{r_n} g(s,y)dy=g(s,r_0).
\end{align}
It follows from \eqref{eq:delat-n1-1}, \eqref{eq:delat-n1-2} and DCT that
\begin{align*}
\lim_{n\to \infty} \Ex\left[\frac{1}{r_n-r_0}\int_{\tau_{r_0}}^{\tau_{r_n}} e^{-\rho s-R_s^{r_0}}dG_s^{h}\right]= \int_{0}^{\infty}\int_{-\infty}^{r_0} e^{-\rho s-r_0+x}\phi_1(s,x,r_0)dx d\Ex[G_s^{h}].
\end{align*}
For the case where $r_0>0$ and $r_n\geq0$ with $r_n\uparrow r_0$ as $n\to \infty$, using a similar argument as above, we can derive that
\begin{align*}
\lim_{n\to \infty}\Delta_n^{(1)}= \beta \int_{0}^{\infty}\int_{-\infty}^{r_0} e^{-\rho s-r_0+x}\phi_1(s,x,r_0)dx d\Ex[G_s^{h}].
\end{align*}
In a similar fashion as in the derivation of \eqref{eq:m-r-1}, we also have
\begin{align*}
\lim_{n\to \infty } \Delta_n^{(2)}&=\beta \lim_{n\to \infty }\Ex\left[\frac{1}{r_n-r_0}\int_{0}^{\tau_{r_0}} e^{-\rho s}\left(e^{-R_s^{r_n}}-e^{-R_s^{r_0}}\right)dG_s^{h}\right]=-\beta \Ex\left[\int_{0}^{\tau_{r_0}} e^{-\rho s-R_s^{r_0}}dG_s^{h}\right].
\end{align*}
At last,  we can also obtain
\begin{align*}
\left|\Delta_n^{(3)}\right|&=\beta\Ex\left[\frac{1}{r_n-r_0}\left|\int_{\tau_{r_0}}^{\tau_{r_n}} e^{-\rho s}\left(e^{-R_s^{r_n}}-e^{-R_s^{r_0}}\right)dG_s^{h}\right|\right]\\
&\leq \beta \Ex\left[\frac{G_{\tau_{r_n}}-G_{\tau_{r_0}}}{r_n-r_0}\sup_{s\geq 0}\left| e^{-\rho s}\left(e^{-R_s^{r_n}}-e^{-R_s^{r_0}}\right)\right|\right].
\end{align*}
Note that $\sup_{s\geq 0}|\frac{e^{-R_s^{r_n}}-e^{-R_s^{r_0}}}{r_n-r_0}|\leq 1$, $\Px$-a.s., and the fact that $G_{\tau_{r_n}} \to G_{\tau_{r_0}}$, a.s., as $n\to \infty$, the DCT yields that $\lim_{n\to \infty}|\Delta_n^{(3)}|=0$. Putting all the pieces together, we get that
\begin{align}\label{eq:m-rr1}
\varphi_{rr}(r_0,h)&= \beta \int_{0}^{\infty}\int_{-\infty}^{r_0} e^{-\rho s-r_0+x}\phi_1(s,x,r_0)dx d\Ex[G_s^{h}]- \beta \Ex\left[\int_{0}^{\tau_{r_0}} e^{-\rho s-R_s^{r_0}}dG_s^{h}\right].
\end{align}

We next derive the representation of the partial derivative $\varphi_h(r,h)$. For any $h_2>h_1 \geq 0$, it follows from \eqref{eq:varphiindependent} that
\begin{align*}
\frac{\varphi\left( r,h_2\right)-\varphi\left(r,h_1\right)}{h_2-h_1}=- \beta\int_0^{\infty} \mathbb{E}\left[e^{-\rho s-R_s^{r}}\right] d\left(\frac{G_s^{h_2}-G_s^{h_1}}{h_2-h_1}\right).
\end{align*}
In lieu of \eqref{eq:L}, it holds that, for $i=1,2$, $G_s^{ h_i}=h_i\vee\{\max _{l\in[0, s]}(\mu_Bl+\sigma_BB^0_{l})\}-h_i$ for $s\geq0$. For $h>0$, we introduce  $\eta_{h} :=\inf \left\{s\geq 0;~  \sigma_B W_s+\mu_B s=h\right\}$ with $\inf \emptyset=+\infty$ by convention. A direct calculation yields that, for all $s \geq 0$,
\begin{align*}
&\frac{G_s^{h_2}-G_s^{h_1}}{h_2-h_1}=
\begin{cases}
\displaystyle ~~~~0, & ~~~s< \eta_{h_1}, \\[0.4em]
\displaystyle-\frac{G_s^{h_1}}{h_2-h_1},  & \eta_{h_1}\leq s<\eta_{h_2},\\[0.6em]
\displaystyle~~~-1, & ~~~s\geq \eta_{h_2}.
\end{cases}
\end{align*}
Then,  it holds that
\begin{align*}
&\mathbb{E}\left[\int_0^{\infty} e^{-\rho s-R_s^{r}}d\left(\frac{G_s^{h_2}-G_s^{h_1}}{h_2-h_1}\right)\right]=-\mathbb{E}\left[\int_{\eta_{h_1}}^{\eta_{h_2}}e^{-\rho s-R_s^{r}} d\left(\frac{G_s^{h_1}}{h_2-h_1}\right)\right]   \nonumber\\
&\qquad=-\mathbb{E}\left[\int_{\eta_{h_1}}^{\eta_{h_2}}\left(e^{-\rho s-R_s^{r}}-e^{-\rho \eta_{h_1}-R_{\eta_{h_1}}^{r}}\right) d\left(\frac{G_s^{h_1}}{h_2-h_1}\right)\right]-\Ex\left[e^{-\rho\eta_{h_1}-R_{\eta_{h_1}}^{r}}\right].
\end{align*}
Note that, as $h_2 \downarrow h_1$, we have
\begin{align*}
&\mathbb{E}\left[\int_{\eta_{h_1}}^{\eta_{h_2}}\!\!\left(e^{-\rho s-R_s^{r}}-e^{-\rho \eta_{h_1}-R_{\eta_{h_1}}^{r}}\right)d\left(\frac{G_s^{h_1}}{h_2-h_1}\right)\right]\!\!\leq\! \Ex\left[\sup_{s\in [\eta_{h_1},\eta_{h_2}]}\left|e^{-\rho s-R_s^{r}}-e^{-\rho \eta_{h_1}-R_{\eta_{h_1}}^{r}}\right|\right]\to 0.
\end{align*}
This yields that
\begin{align}\label{eq:m-h-1}
\lim _{h_2 \downarrow h_1}\Ex\left[\int_0^{\infty} e^{-\rho s-R_s^{r}} d\left(\frac{G_s^{h_2}-G_s^{h_1}}{h_2-h_1}\right)\right]=-\Ex\left[e^{-\rho \eta_{h_1}-R_{\eta_{h_1}}^{r}}\right].
\end{align}
Similarly, using the above argument, we also obtain
\begin{align*}
\lim _{h_2 \uparrow h_1}\frac{\varphi( r,h_2)-\varphi(r,h_1)}{h_2-h_1}=\lim _{h_2 \downarrow h_1} \frac{\varphi(r,h_2)-\varphi(r,h_1)}{h_2-h_1}.
\end{align*}
Thus, we can conclude that, for all $(r,h)\in\R_+^2$,
\begin{align}\label{eq:m-h}
\varphi_h(r,h)=\beta \Ex\left[e^{-\rho \eta_{h}-R_{\eta_{h}}^{r}}\right].
\end{align}
Furthermore, in view of Proposition 2.5 in \cite{Abraham2000}, it follows that
\begin{align}\label{eq:Ex-K}
\Ex[G_s^{h}]=\int_0^s p(0,h;l,0) dl,\quad \forall s\geq 0,
\end{align}
where $p(0,h_0;s,h)=\Px(P_s^{h_0}\in dh)/dh$ is the condition density function of the reflected drifted Brownian motion $P_s^{h_0}$ at time $s\geq 0$ (c.f. \cite{Vee2004}). Hence, we deduce that
\begin{align*}
\varphi(r,h)=-\beta\int_0^{\infty}\Ex\left[e^{-\rho s-R^{r}_s}\right]p(0,h;s,0)ds,\quad \forall (r,h)\in\R_+^2.
\end{align*}
In view that for every fixed $s\in\R_+$, the function $h\to p(0,h;s,0)$ belongs to $C^2(\R_+)$ (c.f. \cite{Vee2004}), we get that, for all $(r,h)\in \R_+^2$,
\begin{align}\label{eq:m-hh}
\varphi_{hh}(r,h)=-\beta\int_0^{\infty}\Ex\left[e^{-\rho s-R^{r}_s}\right]\frac{\partial^2 p(0,h;s,0)}{\partial h^2}ds.
\end{align}
Following \eqref{eq:m-h} and a similar argument as in the proof of \eqref{eq:m-r-1}, we can obtain
\begin{align}\label{eq:m-rh}
\varphi_{rh}(r,h)&=-\beta \Ex\left[e^{-\rho \eta_{h}-R_{\eta_{h}}^{r}}\mathbf{1}_{\eta_h<\tau_r}\right]\nonumber\\
&=-\beta \Ex\left[\exp\left(-\rho \eta_h-r+\tilde{W}_{\eta_h}-\left(\sup_{q\in[0,\eta_h]}\tilde{W}_q-r\right)^+\right)\mathbf{1}_{\sup_{q\in[0,\eta_h]}\tilde{W}_q<r}\right]\nonumber\\
&=-\beta\int_0^{\infty}\int_0^r\int_{-\infty}^y e^{-\rho s-r+x}\phi_1(s,x,y)\phi_2(s,h)dxdyds.
\end{align}
The desired conclusion on the smoothness of $\varphi(r,h)$ in the lemma follows by combining expressions in \eqref{eq:m-r-1}, \eqref{eq:m-rr1}, \eqref{eq:m-h}, \eqref{eq:m-hh} and \eqref{eq:m-rh}.
\end{proof}

We next give the proof of Lemma~\ref{lem:secondtermm}.

\begin{proof}[Proof of Lemma~\ref{lem:secondtermm}]
In lieu of Lemma~\ref{lem:derivative-m}, we first prove that the function $\varphi(r,h)$ defined by \eqref{eq:fcnm} satisfies the Neumann problem \eqref{eq:HJB-m}. In fact, for all $(r,h)\in \R_+^2$, we introduce that
\begin{align*}
\begin{cases}
\displaystyle \hat{B}_t^{r}:=r+\left(\frac{\alpha^2}{2}-\rho\right)t+\alpha B^1_t,\quad \forall t\in\R_+,\\[0.8em]
\displaystyle \tilde{B}_t^{h}:=h-\mu_Bt-\sigma_B B_t^0,\quad \forall t\in\R_+.
\end{cases}
\end{align*}
Here, we recall that $B^1=(B^1_t)_{t\geq 0}$ and $B^0=(B^0_t)_{t\geq 0}$ are two independent scalar Brownian motions, which are given in \eqref{eq:R} and \eqref{eq:H}. For any $\epsilon\in(0,r\wedge h)$, let us define
\begin{align}
\tau_{\epsilon}:=\inf\{t\geq 0;~|\hat{B}_t^{r}-r|\geq \epsilon~\text{or}~|\tilde{B}_t^{h}-h|\geq \epsilon\}.
\end{align}
We can easily check that, for any $\hat{t}\geq 0$, $\hat{R}^{r}_{\hat{t}\land \tau_{\epsilon}}=\hat{B}^{r}_{\hat{t}\land \tau_{\epsilon}}$ and  $P^{h}_{\hat{t}\land \tau_{\epsilon}}=\tilde{B}^{h}_{\hat{t}\land \tau_{\epsilon}}$. Then, by \eqref{eq:H}, \eqref{eq:L} and the strong Markov property, we get that
\begin{align*}
-\beta \Ex\left[ \int_{\hat{t}\wedge\tau_{\epsilon}}^{\infty} e^{-\rho s-R^{r}_s}dG_s^{h}\Big|{\cal F}_{\hat{t}\wedge\tau_{\epsilon}}\right]
=e^{-\rho (\hat{t}\wedge\tau_{\epsilon})}\varphi\left(R^{r}_{\hat{t}\wedge\tau_{\epsilon}},P^{h}_{\hat{t}\wedge\tau_{\epsilon}}\right),
\end{align*}
where $\varphi(r,h)=-\beta\Ex[ \int_0^{\infty} e^{-\rho s-R^{r}_s}dG_s^{h}]$ using \eqref{eq:fcnm}. Consider using Lemma \ref{lem:derivative-m}, we can verify that the function $\varphi$ satisfies  $\varphi_r(0,h)=0$, and $\varphi_h(r,0)=\beta e^{-r}$ for all $(r,h)\in \R_+^2$. Therefore, for $(r,h)\in \R_+^2$, it holds that
\begin{align*}
\varphi(r,h)=\Ex\left[e^{-\rho (\hat{t}\wedge\tau_{\epsilon})}\varphi\left(R^{r}_{\hat{t}\wedge\tau_{\epsilon}},P^{h}_{\hat{t}\wedge\tau_{\epsilon}}\right)
-\beta\int_0^{\hat{t}\wedge\tau_{\epsilon}}e^{-\rho s-R^{r}_s}dG_s^{h}\right].
\end{align*}
By Lemma~\ref{lem:derivative-m} and It{\^o}'s formula, we have
\begin{align}\label{eq:hat-t}
&\frac{\beta}{\hat{t}}\Ex\left[\int_0^{\hat{t}\wedge\tau_{\epsilon}}e^{-\rho s-R^{r}_s}dG_s^{h}\right]
=\frac{1}{\hat{t}}\Ex\left[e^{-\rho (\hat{t}\wedge\tau_{\epsilon})}\varphi\left(R^{r}_{\hat{t}\wedge\tau_{\epsilon}},P^{h}_{\hat{t}\wedge\tau_{\epsilon}}\right)-\varphi(r,h)\right]\nonumber\\
&\quad=\frac{1}{\hat{t}}\Ex\left[\int_0^{\hat{t}\wedge\tau_t^{\epsilon}} e^{-\rho s}\left(\mathcal{L}\varphi-\rho \varphi\right)(R_{s}^{ r},P_{s}^{h})d s\right]+\frac{1}{\hat{t}}\Ex\left[\int_0^{\hat{t}\wedge\tau_{\epsilon}} e^{-\rho s}\varphi_r(R_{s}^{ r},P_{s}^{h})  d L_s^{ r}\right] \nonumber\\
&\qquad+\frac{1}{\hat{t}}\Ex\left[\int_0^{\hat{t}\wedge\tau_{\epsilon}} e^{-\rho s}\varphi_h(R_{s}^{r},P_{s}^{h})  d G_s^{ h}\right],
\end{align}
where the operator ${\cal L}$ is defined on $C^{2}(\R_+^2)$ that
\begin{align*}
{\cal L}g:=\frac{\alpha^2}{2}g_{rr}+\left(\frac{\alpha^2}{2}-\rho\right)g_r+\frac{1}{2}\sigma_B^2 g_{hh}-\mu_Bg_h,\quad \forall g\in C^{2}(\R_+^2).
\end{align*}
The DCT yields that
\begin{align*}
\lim_{\hat{t}\downarrow 0}\frac{1}{\hat{t}}\Ex\left[\int_0^{\hat{t}\wedge\tau_{\epsilon}}e^{-\rho s}\left(\mathcal{L}\varphi-\rho\varphi\right)( R_{s}^{r},P_{s}^{h}) d s\right]=\left(\mathcal{L}\varphi-\rho\varphi\right)(r,h).
\end{align*}
Note that, for all $r\in \R_+$, $\varphi_h(r,0)=\beta e^{-r}$ and $R_s^{r}>0$ on $s\in[0,\hat{t}\wedge\tau_{\epsilon}]$.  We then have
\begin{align*}
\frac{1}{\hat{t}}\Ex\left[\int_{0}^{\hat{t}\wedge\tau_{\epsilon}}\varphi_r(R_{s}^{r},P_{s}^{h})d L_s^{r}\right]=0,\quad  \varphi_h(R_{s}^{r},P_{s}^{h})\mathbf{1}_{\{P_s^{ h}=0\}}=\beta e^{-R_{s}^{r}}.
\end{align*}
By using \eqref{eq:hat-t}, we obtain $\left(\mathcal{L}\varphi-\rho\varphi\right)(r,h)=0$ on $(r,h)\in \R_+^2$.

Next, we assume that the PDE \eqref{eq:HJB-m} admits a classical solution $\varphi$ satisfying $|\varphi(r,h)|\leq C$ for some constant $C>0$ depending on $(\mu,\sigma,\mu_B,\sigma_B)$ only. Using the Neumann problem \eqref{eq:HJB-m}, the It{\^o}'s formula gives that, for all $(T,r,h) \in \R_+^3$,
\begin{align}\label{eq:HJB-m-cla-sol}
&\Ex\left[e^{-\rho T}\varphi(R_{T}^{ r},P_{T}^{h})\right]=\varphi(r,h)+\Ex\left[\int_0^{T}e^{-\rho s}\left(\mathcal{L}\varphi-\rho \varphi\right)(R_{s}^{ r},P_{s}^{h}) d s\right]\nonumber \\
&\qquad+\Ex\left[\int_0^{T} e^{-\rho s}\varphi_r(R_{s}^{r},P_{s}^{h}) \mathbf{1}_{\{R_s^{ r}=0\}} d L_s^{r}\right]+\Ex\left[\int_0^{T} e^{-\rho s}\varphi_h(R_{s}^{r},P_{s}^{h}) \mathbf{1}_{\{P_s^{ h}=0\}} d G_s^{h}\right]\nonumber\\
&\quad= \varphi(r, h)+\beta \Ex\left[\int_0^{T} e^{-R_s^{ r}-\rho s} d G_s^{h}\right].
\end{align}
Moreover, by DCT and the boundedness of $\varphi$ on $\R_+^2$, $\lim_{T\to \infty}\Ex\left[e^{-\rho T}\varphi(R_{T}^{ r},P_{T}^{h})\right]=0$. Letting $T \rightarrow \infty$ on both sides of \eqref{eq:HJB-m-cla-sol} and using DCT and MCT, we obtain the representation \eqref{eq:fcnm} of the solution $\varphi(r,h)$, which completes the proof.
\end{proof}

Next, we propose a homogenization method of Neumann boundary conditions to study the smoothness of the last term in the probabilistic representation \eqref{eq:v} together with the application of the result obtained in Lemma~\ref{lem:secondtermm}.
\begin{proposition}\label{prop:secondtermm}
For any $(r,h)\in \R_+^2$, define the function $\xi(r,h)$ as follows:
\begin{align}\label{eq:fcnmxi}
\xi(r,h):=-\kappa_1 \Ex\left[\int_0^{\infty}e^{-\rho s}\varphi_{rh}(R_s^r,H_s^h)ds\right].
\end{align}
Here, we recall that the reflected processes $R^{r}=(R_t^{r})_{t\geq 0}$ and $H^{h}=(H_t^{h})_{t\geq 0}$ with $(r,h)\in\R_+^2$ are given in \eqref{eq:R} and \eqref{eq:H}, respectively. Moreover, let us define that
\begin{align}\label{eq:psi}
\psi(r,h):=\varphi(r,h)+ \xi(r,h),\quad \forall (r,h)\in\R_+,
\end{align}
where the function $\varphi(r,h)\in C^2(\R_+^2)$ is given by \eqref{eq:varphiindependent} in Lemma~\ref{lem:secondtermm}. Then, the function $\psi(r,h)$ is a classical solution to the following Neumann problem with  Neumann boundary conditions at $r=0$ and $h=0$:
\begin{align}\label{eq:HJB-psi}
\begin{cases}
\displaystyle \frac{\alpha^2}{2}\psi_{rr}+\left(\frac{\alpha^2}{2}-\rho\right)\psi_r+\frac{\sigma_B^2}{2} \psi_{hh}-\mu_B \psi_h-\kappa_1 \psi_{rh}=\rho \psi,~ \text{on}~(0,\infty)^2,\\[0.8em]
\displaystyle \psi_r(0,h)=0,\quad \forall h\in\R_+,\\[0.8em]
\displaystyle \psi_h(r,0)=\beta e^{-r},\quad \forall r\in\R_+.
\end{cases}
\end{align}
On the other hand, if the Neumann problem \eqref{eq:HJB-psi} has a classical solution $\psi(r,h)$ satisfying $|\psi(r,h)|\leq C$ for some constant $C>0$ depending on $(\mu,\sigma,\mu_B,\sigma_B,\gamma)$, then this solution $\psi(r,h)$ satisfies the following probabilistic representation:
\begin{align}\label{eq:fcnm}
\psi(r,h)=-\beta\Ex\left[\int_0^{\infty}e^{-\rho s-R^{r}_s}dK_s^{h}\right].
\end{align}
\end{proposition}

\begin{proof}
Note that the integral in the expectation is the Lebesgue integral in \eqref{eq:fcnmxi}. Then, using the stochastic flow argument  (c.f. the argument used in Theorem~4.2 in \cite{BoLiaoYu21}), it is not difficult to verify that $\xi(r,h)\in C^{2}(\R_+^2)$ is a classical solution to the following Neumann problem with homogeneous Neumann boundary conditions:
\begin{align}\label{eq:HJB-xi}
\begin{cases}
\displaystyle \frac{\alpha^2}{2}\xi_{rr}+\left(\frac{\alpha^2}{2}-\rho\right)\xi_r+\frac{\sigma_B^2}{2} \xi_{hh}-\mu_B \xi_h-\kappa_1 \xi_{rh}-\rho \xi=\kappa_1 \varphi_{rh},~ \text{on}~(0,\infty)^2,\\[0.9em]
\displaystyle \xi_r(0,h)=0,\quad \forall h\in\R_+,\\[0.8em]
\displaystyle \xi_h(r,0)=0,\quad \forall r\in\R_+.
\end{cases}
\end{align}
Using Eq.~\eqref{eq:HJB-m} in Lemma~\ref{lem:secondtermm}, i.e.,
\begin{align*}
\begin{cases}
\displaystyle \frac{\alpha^2}{2}\varphi_{rr}+\left(\frac{\alpha^2}{2}-\rho\right)\varphi_r+\frac{\sigma_B^2}{2} \varphi_{hh}-\mu_B \varphi_h=\rho \varphi,~ \text{on}~(0,\infty)^2,\\[0.8em]
\displaystyle \varphi_r(0,h)=0,\quad \forall h\in\R_+,\\[0.8em]
\displaystyle \varphi_h(r,0)=\beta e^{-r},\quad \forall r\in\R_+.
\end{cases}
\end{align*}
In terms of the above two Neumann problem, we can conclude that $\psi(r,h)=\varphi(r,h)+ \xi(r,h)$ satisfies the Neumann problem \eqref{eq:HJB-psi}. This shows the first part of the proposition.

We next prove the second part of the proposition. To do it, we assume that the Neumann problem \eqref{eq:HJB-m} admits a classical solution $\psi$ satisfying $|\psi(r,h)|\leq C$ for some constant $C>0$ depending on $(\mu,\sigma,\mu_B,\sigma_B,\gamma)$ only. Then, the It{\^o}'s formula gives that, for all $(T,r,h) \in \R_+^3$,
\begin{align}\label{eq:HJB-m-cla-sol}
&\Ex\left[e^{-\rho T}\psi(R_{T}^{ r},H_{T}^{h})\right]=\psi(r,h)+\Ex\left[\int_0^{T}e^{-\rho s}\left(\mathcal{L}\psi-\rho \psi\right)(R_{s}^{ r},H_{s}^{h}) d s\right]\nonumber \\
&\qquad+\Ex\left[\int_0^{T} e^{-\rho s}\psi_r(R_{s}^{r},H_{s}^{h}) \mathbf{1}_{\{R_s^{ r}=0\}} d L_s^{r}\right]+\Ex\left[\int_0^{T} e^{-\rho s}\psi_h(R_{s}^{r},H_{s}^{h}) \mathbf{1}_{\{H_s^{ h}=0\}} d K_s^{h}\right],
\end{align}
where the operator ${\cal L}$ is defined on $C^{2}(\R_+^2)$ that
\begin{align*}
{\cal L}g:=\frac{\alpha^2}{2}g_{rr}+\left(\frac{\alpha^2}{2}-\rho\right)g_r+\frac{\sigma_B^2}{2} g_{hh}-\mu_Bg_h-\kappa_1 g_{rh},\quad \forall g\in C^{2}(\R_+^2).
\end{align*}
It follows from the first equation in \eqref{eq:HJB-psi} that $\Ex[\int_0^{T}e^{-\rho s}(\mathcal{L}\psi-\rho \psi)(R_{s}^{ r},H_{s}^{h})ds]=0$. Using the Neumann boundary conditions in \eqref{eq:HJB-psi}, we obtain
\begin{align*}
\Ex\left[\int_0^{T} e^{-\rho s}\psi_r(R_{s}^{r},H_{s}^{h}) \mathbf{1}_{\{R_s^{ r}=0\}} d L_s^{r}\right]=\Ex\left[\int_0^{T} e^{-\rho s}\psi_r(0,H_{s}^{h})d L_s^{r}\right]=0,
\end{align*}
and
\begin{align*}
\Ex\left[\int_0^{T} e^{-\rho s}\psi_h(R_{s}^{r},H_{s}^{h}) \mathbf{1}_{\{H_s^{ h}=0\}} d K_s^{h}\right]=\Ex\left[\int_0^{T} e^{-\rho s}\psi_h(R_{s}^{r},0) d K_s^{h}\right]=\beta\Ex\left[\int_0^{T} e^{-\rho s-R_s^r}d K_s^{h}\right].
\end{align*}
This yields from \eqref{eq:HJB-m-cla-sol} that
\begin{align}\label{eq:HJB-m-cla-sol2}
\Ex\left[e^{-\rho T}\psi(R_{T}^{ r},H_{T}^{h})\right]=\psi(r, h)+\beta \Ex\left[\int_0^{T} e^{-R_s^{ r}-\rho s} d K_s^{h}\right].
\end{align}
Moreover, we have the boundedness of $\psi$ on $\R_+^2$. In fact, the function $\varphi$ is bounded on $\R_+^2$ via \eqref{eq:varphiindependent}; while $\varphi_{rh}$ is also bounded by applying \eqref{eq:m-rh} in Lemma~\ref{lem:derivative-m}. This gives from DCT that $\lim_{T\to \infty}\Ex\left[e^{-\rho T}\psi(R_{T}^{ r},H_{T}^{h})\right]=0$. Letting $T \rightarrow \infty$ on both sides of \eqref{eq:HJB-m-cla-sol}, using DCT and MCT, we obtain the representation \eqref{eq:fcnm} of the solution $\psi(r,h)$, which completes the proof.
\end{proof}

\begin{remark}\label{rem:psi}
Proposition \ref{prop:secondtermm} has showed that, the function $\psi$ given by \eqref{eq:fcnm} is in $C^2(\R_+^2)$. Furthermore, it also holds that
\begin{align}
\psi_r( r, h)
&=\beta\Ex\left[\int_0^{\tau_{r} } e^{-\rho s-R_s^{r}} d K_s^{h}\right],\quad \psi_{h}(r,h)=\beta \Ex\left[e^{-\rho \zeta_{h}-R^{r}_{ \zeta_{h}}}\right],\label{eq:psi-r}\\
\psi_{rr}(r,h)&= \beta\lim_{\Delta r\to 0}\Ex\left[\frac{1}{\Delta r}\int_{\tau_{r}}^{\tau_{r+\Delta r}} e^{-\rho s-R_s^{r}}dK_s^{h}\right]- \beta \Ex\left[\int_{0}^{\tau_{r}} e^{-\rho s-R_s^{r}}dK_s^{h}\right],\label{eq:psi-rr}\\
\psi_{rh}(r,h)&=-\beta \Ex\left[e^{-\rho \zeta_{h}-R_{\zeta_{h}}^{r}}\mathbf{1}_{\zeta_h<\tau_r}\right].\label{eq:psi-rh}
\end{align}
Here, $\tau_r$ is defined in Lemma \ref{lem:derivative-m}, and $\zeta_{h}:=\inf\{s \geq 0;~\sigma_B B^3_s+\mu_B s=h\}$ with convention $\inf\emptyset=+\infty$. From these representations, it follows that $\psi_r(r,h)+\psi_{rr}(r,h)>0$ for all $(r,h)\in\R_+^2$.
\end{remark}

We can finally present the proof of the main result in this section, i.e., Theorem~\ref{thm:dualPDEclass}.
\begin{proof}[Proof of  Theorem~\ref{thm:dualPDEclass}]
By applying Lemma~\ref{lem:v1stpart} and  Proposition~\ref{prop:secondtermm}, the function $v(r,h,z)$ defined by \eqref{eq:v} is a classical solution to the following Neumann problem:
\begin{align}\label{eq:HJB-v}
\begin{cases}
\displaystyle \frac{\alpha^2}{2}v_{rr}+\left(\frac{\alpha^2}{2}-\rho\right) v_r+\frac{1}{2}\sigma_B^2 v_{hh}-\mu_B v_h+\frac{1}{2}\sigma_Z^2 z^2 v_{zz}+\mu_Z z v_z -\kappa_1 v_{rh}+\kappa_2z v_{rz}\\[0.8em]
\displaystyle \qquad\qquad+\sigma_Z\sigma_B \eta^{\top}\gamma zv_{zh}+(\kappa_2-\mu_Z)\beta z e^{-r}+\left(\frac{1-p}{p}\right)\beta^{-\frac{p}{1-p}}e^{\frac{p}{1-p}r}=\rho v,~\text{on}~(0,\infty)^3,\\[0.8em]
\displaystyle v_r(0,h,z)=0,\quad \forall (h,z)\in \R_+^2,\\[0.8em]
\displaystyle v_h(r,0,z)= \beta e^{-r},\quad \forall (r,z)\in \R_+^2.
\end{cases}
\end{align}
Then, we can verify that $\hat{u}(y,h,z)= v(-\ln \frac{y}{\beta},h,z)$ for $(y,h,z)\in(0,\beta]\times\R_+^2$ is a classical solution of the Neumann problem \eqref{eq:dual-u} with Neumann boundary conditions \eqref{b-v-1} and \eqref{b-v-2}. Moreover, the strict convexity of $(0,\beta] \ni y \rightarrow\hat{u}(y,h,z)$ for fixed $(h,z) \in \R_+^2$ follows from the fact that $\hat{u}_{y y}=\frac{1}{y^2}\left[v_{rr}+v_r\right]>0$ by applying  Lemma~\ref{lem:v1stpart}, Lemma \ref{lem:derivative-m} and Remark \ref{rem:psi}. On the other hand, in a similar fashion of  Proposition \ref{prop:secondtermm}'s proof, we can verify that if the Neumann problem \eqref{eq:dual-u}-\eqref{b-v-2} has a classical solution $\hat{u}(y,h,z)$ satisfying $|\hat{u}(y,h,z)| \leq C(1+|y|^{-q}+z^q)$ for some $q>1$ and some constant constant $C>0$, then $v(r,h,z):= \hat{u}(\beta e^{-r},h,z)$ has the probabilistic representation \eqref{eq:v}.
\end{proof}

\hypertarget{sec:verfication}{%
\section{Verification Theorem}\label{sec:verfication}}

Theorem \ref{thm:dualPDEclass} shows existence and uniqueness of the classical solution $\hat{u}(y,h,z)$ for $(y,h,z)\in  (0,\beta]\times \R_+^2$ to the dual PDE \eqref{eq:dual-u} with two Neumann boundary conditions \eqref{b-v-1} and \eqref{b-v-2}. Moreover, this solution $\hat{u}(y,h,z)$ is strictly convex in $y \in (0, \beta]$. The following verification theorem will recover the classical solution $u(x,h,z)$ of the primal HJB equation \eqref{HJB} via the inverse transform of $\hat{u}(y,h,z)$, and provide the optimal (admissible) portfolio-consumption control in the feedback form to the primal stochastic control problem \eqref{eq_mfg-2}.

\begin{theorem}\label{thm:verification}
Let $\rho_0>0$  be the constant depending on model parameters $(\mu,\sigma,\mu_B,\sigma_B,\mu_Z,\sigma_Z,\gamma,p,\beta)$ explicitly specified later in \eqref{boundrho}. For the discount rate $\rho>\rho_0$, it holds that:
\begin{itemize}
\item[{\rm(i)}] Consider the function $v(r,h,z)$ for $(r,h,z)\in\R_+^3$ defined by the probabilistic representation \eqref{eq:v}. Let $\hat{u}(y,h,z)=v(-\ln\frac{y}{\beta},h,z)$  for all $(y,h,z)\in (0,\beta]\times \R_+^2$. For any $(x,h,z)\in\R_+^3$, introduce that
\begin{align}\label{eq:inversedual-infinite}
u(x,h,z)=\inf_{y\in(0,\beta]}\{\hat{u}(y,h,z)+yx\}.
\end{align}
Then, the function $u(x,h,z)$ is a classical solution to the following HJB equation with Neumann boundary conditions:
\begin{align}\label{HJB-infinite}
\begin{cases}
\displaystyle \sup_{\theta\in\R^d}\left[\theta^{\top}\mu u_x+\frac{1}{2}\theta^{\top}\sigma\sigma^{\top}\theta u_{xx}+\sigma_Z \theta^{\top}\sigma\eta z(u_{xx}-u_{xz})-\sigma_B \theta^{\top}\sigma\gamma u_{xh}\right]\\[1em]
\displaystyle\quad+\sup_{c\geq0}\left(\frac{c^p}{p}-cu_x\right) +\frac{1}{2}\sigma_B^2u_{hh} -\mu_B u_h+\frac{1}{2}\sigma_Z^2z^2(u_{zz}+u_{xx}-2u_{xz})\\[1em]
\displaystyle\qquad\qquad +\mu_Z z(u_z-u_x)+\sigma_Z\sigma_B z\eta^{\top}\gamma (u_{xh}-u_{hz})=\rho u,\\[1em]
\displaystyle u_x(0,h,z)=\beta,\quad \forall (h,z)\in\R_+^2,\\[0.6em]
\displaystyle u_h(x,0,z)= u_x(x,0,z),\quad \forall (x,z)\in\R_+^2.
\end{cases}
\end{align}
\item[{\rm(ii)}] Define the following optimal feedback control function by, for all $(x,h,z)\in\R_+^3$,
\begin{align}\label{eq:feedbackfcn}
\theta^*(x,h,z):=-(\sigma\sigma^{\top})^{-1}\frac{\mu u_x-\sigma_B\sigma\gamma u_{xh}+\sigma_Z \sigma\eta z(u_{xx}-u_{xz})}{u_{xx}}, \quad c^*(x,h,z):=u_{x}^{\frac{1}{p-1}}.
\end{align}
With $(x,h,z)\in\R_+^3$, consider the controlled reflected process $(X^*,I,Z)=(X_t^*,I_t,Z_t)_{t\geq 0}$ given by, for all $t\geq 0$,
\begin{align}\label{eq:optimal-SDE-XZ}
\begin{cases}
\displaystyle X_t^* = x+\int_0^t\theta^*(X^*_{s},I_s,Z_{s})^{\top}\mu ds+\int_0^t\theta^*(X^*_{s},I_s,Z_{s})^{\top}\sigma dW_{s}-\int_0^t c^*(X_{s}^*,I_s,Z_{s})ds\\[0.6em]
\displaystyle \qquad\quad -\int_0^t \mu_Z Z_sds-\int_0^t \sigma_Z Z_sdW^{\eta}_s-\int_0^t dm_{s}+L_t^{X^*},\\[0.6em]
\displaystyle I_t=h-\int_0^t \mu_Bds-\int_0^t \sigma_BdW^{\gamma}_{s}+\int_0^t dm_{s},\\[0.6em]
\displaystyle Z_t=z+\int_0^t \mu_Z Z_sds+\int_0^t \sigma_Z Z_sdW^{\eta}_s.
\end{cases}
\end{align}
Above, the running maximum process $m=(m_t)_{t\geq0}$ is given in \eqref{eq:Mtn} and $L^{X^*}_0=0$. Define $\theta_t^*=\theta^*(X^*_{t},I_t,Z_{t})$ and $c_t^*=c^*(X_{t}^*,I_t,Z_{t})$ for all $t\geq0$. Then, $(\theta^*,c^*)=(\theta_t^*,c_t^*)_{t\geq 0}\in\mathbb{U}^{\rm r}$ is an optimal investment-consumption strategy. Moreover, for any admissible strategy $(\theta,c)\in\mathbb{U}^{\rm r}$, we have
\begin{align*}
\Ex\left[\int_0^{\infty} e^{-\rho t} \frac{(c_t)^p}{p}dt- \beta \int_0^{\infty} e^{-\rho t}dL_t^X \right]\leq u(x,h,z),\quad \text{for all}~ (x,h,z)\in\R_+^3,
\end{align*}
where the equality holds when $(\theta,c)=(\theta^*,c^*)$. 
\end{itemize}
\end{theorem}

\begin{proof}
We first prove the item (i). For $(x,h,z)\in\R_+^3$, let us define $y^*=y^*(x,h,z)\in(0,\beta]$ satisfying $\hat{u}_y(y^*,h,z)=-x$. Then, we have
\begin{align}\label{eq:dual-u2}
u(x,h,z)=\inf_{y\in(0,\beta]}\{\hat{u}(y,h,z)+yx\}=\hat{u}(y^*(x,h,z),h,z)+xy^*(x,h,z).
\end{align}
By applying Lemma~\ref{lem:v1stpart} and  Lemma~\ref{lem:derivative-m}, it follows that
\begin{align}\label{eq:dual-u_y}
\hat{u}_y(y,h,z)=-\frac{1}{ y} v_r\left(-\ln \frac{y}{\beta},h,z\right)=-\frac{1}{y}\left[l_r\left(-\ln \frac{y}{\beta},z\right)+\psi_r\left(-\ln \frac{y}{\beta},h\right)\right]\leq0.
\end{align}
 Then, $(0,\beta]\ni y\to \hat{u}(y,h,z)$ is decreasing for fixed $(z,h)\in \R_+^2$. Moreover, note that $\hat{u}_y(\beta,h,z)=0$, and hence
\begin{align*}
\lim_{y\to 0}\hat{u}_y(y,h,z)=\lim_{r\to +\infty}\hat{u}_y(\beta e^{-r},h,z)=-\lim_{r\to +\infty}e^{r}v_r(r,h,z)=-\infty.
\end{align*}
Thus, $y^*$ and $u$ defined by \eqref{eq:inversedual-infinite} is well-defined on $\R_+^2$.  Moreover, it follows from Theorem \ref{thm:dualPDEclass} that $y \mapsto \hat{u}(y,h,z)$ is strictly convex, which implies that $x \mapsto u(x,h,z)$ is strictly concave.
Thus, a direct calculation yields that $u$ solves the primal HJB equation \eqref{HJB}.

We next prove the item (ii). It follows from Theorem \ref{thm:dualPDEclass} that  $\theta^*(x,h,z)$ and $c^*(x,h,z)$ given by \eqref{eq:feedbackfcn} are continuous on $\R_+^3$. We then claim that, there exists a pair of positive constants $(C_o,C_q)$ such that, for all $(x,h,z)\in \R_+^3$,
\begin{align}\label{eq:op-control-growth}
|\theta^*(x,h,z)|\leq C_{o}(1+x+z),\quad |c^*(x,h,z)|\leq C_{q}(1+x).
\end{align}
Let $\gamma_1:=|(\sigma\sigma^{\top})^{-1}\mu|$, $\gamma_2:=|(\sigma\sigma^{\top})^{-1}\sigma_B\sigma\gamma|$ and $\gamma_3:=|(\sigma\sigma^{\top})^{-1}\sigma_Z\sigma\eta|$. In view of the duality transform, we arrive at
\begin{align}\label{op-theta}
&|\theta^*(x,z)|\leq \gamma_1 \left|\frac{u_x}{u_{xx}}(x,h,z)\right|+\gamma_2\left|\frac{u_{xz}}{u_{xx}}(x,h,z)\right|+\gamma_3 \left|\frac{zu_{xz}}{u_{xx}}(x,h,z)\right|+\gamma_3 z\nonumber\\
&=\gamma_1 y^*(x,h,z)\hat{u}_{yy}(y^*(x,h,z),h,z)+\gamma_2 |\hat{u}_{yh}(y^*(x,h,z),h,z)|+\gamma_3 |z\hat{u}_{yz}(y^*(x,h,z),h,z)|+\gamma_3 z\nonumber\\
&= \gamma_1 y^*(x,h,z)v_{rr}\left(-\ln\frac{y^*(x,h,z)}{\beta},h,z\right)+\gamma_2 \left|v_{rh}\left(-\ln\frac{y^*(x,h,z)}{\beta},h,z\right)\right|\nonumber\\
&\quad+\gamma_3 \left|zv_{rz}\left(-\ln\frac{y^*(x,h,z)}{\beta},h,z\right)\right|+\gamma_3 z\nonumber\\
&=\gamma_1 x+\frac{\gamma_1 }{y^*(x,h,z)}(l_{rr}+\psi_{rr})\left(-\ln\frac{y^*(x,h,z)}{\beta},h,z\right)+\frac{\gamma_2}{y^*(x,h,z)}\left|\psi_{rh}\left(-\ln\frac{y^*(x,h,z)}{\beta},h\right)\right|\nonumber\\
&\quad+\frac{\gamma_3}{y^*(x,h,z)} \left|zl_{rz}\left(-\ln\frac{y^*(x,h,z)}{\beta},z\right)\right|+\gamma_3 z,
\end{align}
where the last equality holds since $x=-\hat{u}_y(y^*(x,h,z),h,z)=\frac{1 }{y^*(x,h,z)}v_r\left(-\ln\frac{y^*(x,h,z)}{\beta},h,z\right)$$=\frac{1 }{y^*(x,h,z)}(l_r+\psi_r)\left(-\ln\frac{y^*(x,h,z)}{\beta},h,z\right)$. It follows from \eqref{eq:psi-r} and \eqref{eq:psi-rr} that, for all $(r,h)\in\R_+^2$,
\begin{align}\label{op-theta-2}
\psi_r(r,h)+\psi_{rr}(r,h)&= \beta\lim_{\Delta r\to 0}\Ex\left[\frac{1}{\Delta r}\int_{\tau_{r}}^{\tau_{r+\Delta r}} e^{-\rho s-R_s^{r}}dK_s^{h}\right].
\end{align}
Using the representation \eqref{eq:L}, we get that, for all $h\in\R_+$ and $t\geq s\geq 0$,
\begin{align*}
 K_t^{ h}-K_s^h &=0 \vee\left\{-h+\max_{l\in[0, t]}\left(\mu_Bl+\sigma_BB^2_{l}\right)\right\}-0 \vee\left\{-h+\max_{l\in[0, s]}\left(\mu_Bl+\sigma_BB^2_{l}\right)\right\}\nonumber\\
&=h \vee\left\{\max_{l\in[0, t]}\left(\mu_Bl+\sigma_BB^2_{l}\right)\right\}-h \vee\left\{\max_{l\in[0, s]}\left(\mu_Bl+\sigma_BB^2_{l}\right)\right\}.
\end{align*}
If $\max_{l\in[0, t]}(\mu_Bl+\sigma_BB^2_{l})\leq h$, it holds that $K_t^{ h}-K_s^h=0\leq K_t^0-K_s^0$; while, if $\max_{l\in[0, t]}(\mu_Bl+\sigma_BB^0_{l})> h$, we also  have
\begin{align}
K_t^{ h}-K_s^h &\leq \max_{l\in[0, t]}\left(\mu_Bl+\sigma_BB^2_{l}\right)-h \vee\left\{\max_{l\in[0, s]}\left(\mu_Bl+\sigma_BB^2_{l}\right)\right\}\nonumber\\
&\leq \max_{l\in[0, t]}\left(\mu_Bl+\sigma_BB^2_{l}\right)-\max_{l\in[0, s]}\left(\mu_Bl+\sigma_BB^2_{l}\right)=K_t^{0}-K_s^0.
\end{align}
Hence, we can deduce that $K_t^0-K_t^h\geq K_s^0-K_s^h$, i.e., the process $\{K_t^0-K_t^h\}_{t\geq 0}$ is non-decreasing. This implies that, for all $h\geq 0$,
\begin{align}\label{op-theta-3}
&\psi_r(r,h)+\psi_{rr}(r,h)=\beta\lim_{\Delta r\to 0}\Ex\left[\frac{1}{\Delta r}\int_{\tau_{r}}^{\tau_{r+\Delta r}} e^{-\rho s-R_s^{r}}dK_s^{h}\right]=\beta\lim_{\Delta r\to 0}\Ex\left[\frac{1}{\Delta r}\int_{\tau_{r}}^{\tau_{r+\Delta r}} e^{-\rho s-R_s^{r}}dK_s^{0}\right]\nonumber\\
&\qquad-\beta\lim_{\Delta r\to 0}\Ex\left[\frac{1}{\Delta r}\int_{\tau_{r}}^{\tau_{r+\Delta r}} e^{-\rho s-R_s^{r}}d(K_s^{0}-K_s^h)\right]\leq \beta\lim_{\Delta r\to 0}\Ex\left[\frac{1}{\Delta r}\int_{\tau_{r}}^{\tau_{r+\Delta r}} e^{-\rho s-R_s^{r}}dK_s^{0}\right]\nonumber\\
&\quad=\psi_r(r,0)+\psi_{rr}(r,0).
\end{align}
Note that, it follows from Proposition \ref{prop:secondtermm} that
\begin{align}\label{eq:psi-varphi}
&\psi_{rr}\left(-\ln\frac{y^*(x,h,z)}{\beta},h\right)\leq (\psi_r+\psi_{rr})\left(-\ln\frac{y^*(x,h,z)}{\beta},0\right)-\psi_r\left(-\ln\frac{y^*(x,z)}{\beta},h\right)\nonumber\\
&\qquad\leq (\psi_r+\psi_{rr})\left(-\ln\frac{y^*(x,h,z)}{\beta},0\right)=(\varphi_r+\varphi_{rr}+\xi_r+\xi_{rr})\left(-\ln\frac{y^*(x,h,z)}{\beta},0\right).
\end{align}
In view of Lemma \ref{lem:derivative-m},  we obatin that
\begin{align}
\varphi_r(r,0)+\varphi_{rr}(r,0)&=  \beta \int_{0}^{\infty}\int_{-\infty}^{r} e^{- \rho s-r+x}\phi_1(s,x,r)dx d\Ex\left[G_s^{0}\right].
\end{align}
For $r\in[0,1]$, by the continuity of $r\to\int_{0}^{\infty}\int_{-\infty}^{r} e^{-\rho s+x}\phi_1(s,x,r)dx d\Ex[G_s^{0}]$, we can obtain
\begin{align}\label{op-theta-4}
\int_{0}^{\infty}\int_{-\infty}^{r} e^{-\rho s+x}\phi_1(s,x,r)dx d\Ex[G_s^{0}]&\leq \max_{r\in[0,1]}
\int_{0}^{\infty}\int_{-\infty}^{r} e^{-\rho s+x}\phi_1(s,x,r)dx d\Ex[G_s^{0}]<+\infty.
\end{align}
For the other case $r>1$, we obtain from \eqref{eq:d} that
\begin{align}\label{op-theta-5}
&\int_{0}^{1}\int_{-\infty}^{r} e^{-\rho s+x}\phi_1(s,x,r)dx d\Ex[G_s^{0}]\nonumber\\
&\quad=\int_{0}^{1}\int_{-\infty}^{r} e^{-\rho s+x}\frac{2(2 r-x)}{\sqrt{2\hat{\sigma}^2 \pi s^3}}\exp\left(\frac{\hat{\mu}}{\hat{\sigma}} x-\frac{1}{2} \hat{\mu}^2 s-\frac{(2 r-x)^2}{2\hat{\sigma}^2 s}\right)dxd\Ex[G_s^{0}].\nonumber\\
&\overset{y=r-x}{=}\int_{0}^{1}\int_{0}^{\infty} e^{-\rho s+r-y}\frac{2( r+y)}{\sqrt{2\hat{\sigma}^2 \pi s^3}}\exp\left(\frac{\hat{\mu}}{\hat{\sigma}}(r-y)-\frac{1}{2}\hat{\mu}^2s-\frac{( r+y)^2}{2\hat{\sigma}^2 s}\right)dyd\Ex[G_s^{0}]\nonumber\\
&\quad\leq \int_{0}^{1}e^{-\rho s}\frac{2\exp\left(-\frac{1}{2\hat{\sigma}^2 s}\right)}{\sqrt{2\hat{\sigma}^2 \pi s^3}}\int_{0}^{\infty} ( r+y)\exp\left(\left(\frac{\hat{\mu}}{\hat{\sigma}}+1\right)r+\frac{y^2}{2\hat{\sigma}^2s}-\frac{( r+y)^2-1}{2\hat{\sigma} ^2s}\right)dyd\Ex[G_s^{0}]\nonumber\\
&\quad\leq 9\hat{\sigma}^{\frac{11}{2}} \int_{0}^{1}e^{-\rho s}d\Ex[G_s^0]\int_{0}^{\infty} ( r+y)\exp\left(\left(\frac{\hat{\mu}}{\hat{\sigma}}+1\right)r-\frac{y-1}{\hat{\sigma}^2 }\right)dy\nonumber\\
&\quad\leq  9\hat{\sigma}^{\frac{11}{2}}  \Ex\left[\int_0^1 dG_s^{0}\right]\int_{0}^{\infty} ( 1+y)\exp\left(-\frac{y-1}{\hat{\sigma}^2 }\right)dy\leq 18\hat{\sigma}^5e^{\hat{\sigma}^{-2}}(|\mu_B|+3\sigma_B).
\end{align}
Moreover, using the fact $\frac{1}{2} \hat{\mu}^2 s+\frac{(2 r-x)^2}{2\hat{\sigma}^2 s}\geq \frac{\hat{\mu}}{\hat{\sigma}} (r-x)$, it holds that
\begin{align}\label{op-theta-6}
 &\int_{1}^{\infty}\int_{-\infty}^{r} e^{-\rho s+x}\frac{2(2 r-x)}{\sqrt{2\hat{\sigma}^2 \pi s^3}}\exp\left(\frac{\hat{\mu}}{\hat{\sigma}} x-\frac{1}{2} \hat{\mu}^2 s-\frac{(2 r-x)^2}{2\hat{\sigma}^2 s}\right)dx d\Ex[G_s^{0}]\nonumber\\
&\qquad\leq \int_{1}^{\infty}\int_{-\infty}^{r} e^{-\rho s+x}\frac{2(2 r-x)}{\sqrt{2\hat{\sigma}^2 \pi }}\exp\left(\frac{2\hat{\mu}}{\hat{\sigma}} r\right)dx d\Ex[G_s^{0}]\nonumber\\
&\qquad\overset{y=r-x}{=\!\!=\!\!=\!\!=\!\!=}\int_{1}^{\infty}\int_{0}^{\infty} e^{-\rho s+r-y}\frac{2(r+y)}{\sqrt{2\hat{\sigma}^2 \pi }}\exp\left(\frac{2\hat{\mu}}{\hat{\sigma}} r\right)dy d\Ex[G_s^{0}]\nonumber\\
&\qquad\leq  \frac{4}{\sqrt{2\hat{\sigma}^2 \pi }}\int_{0}^{\infty} e^{-\rho s} d\Ex[G_s^{0}]\leq \frac{4(|\mu_B|+3\sigma_B)}{\sqrt{2\hat{\sigma}^2 \pi }}.
\end{align}
Thus, it follows from \eqref{op-theta-2}-\eqref{op-theta-6}  that
\begin{align}\label{op-theta-m}
\frac{1}{y^*(x,z)}(\varphi_{rr}+\varphi_r)\left(-\ln\frac{y^*(x,z)}{\beta},0\right)\leq M_1(1+x),
\end{align}
where the positive constant $M_1$ is defined by
\begin{align}\label{eqM1}
M_1:=18\hat{\sigma}^5e^{\hat{\sigma}^{-2}}(|\mu_B|+3\sigma_B)+\frac{4(|\mu_B|+3\sigma_B)}{\sqrt{2\hat{\sigma}^2 \pi}}+ \max_{r\in[0,1]}\int_{0}^{\infty}\int_{-\infty}^{r} e^{-s+x}\phi_1(s,x,r)dx d\Ex[G_s^{0}].
\end{align}

In what follows, let us define that, for $(r,h)\in\R_+^2$,
\begin{align*}
f(r,h)&:=\varphi_{rh}(r,h)\!\!=\!\!-\beta \Ex\left[e^{-\rho \eta_{h}-R_{\eta_{h}}^{r}}\mathbf{1}_{\eta_h<\tau_r}\right]\!\!=\!\!-\beta\int_0^{\infty}\int_0^r\int_{-\infty}^y e^{-\rho s-r+x}\phi_1(s,x,y)\phi_2(s,h)dxdyds.
\end{align*}
Using Proposition \ref{prop:secondtermm} and Proposition 4.1 in \cite{BoLiaoYu21}, we obtain that
\begin{align}
\xi_{r}(r,0)&=-\kappa \Ex\left[\int_0^{\tau_r}e^{-\rho s}f_r(R_s^r,H_s^0)ds\right],\label{eq:xi-r}\\
\xi_{rr}(r,0)&=-\kappa\Gamma\Ex\left[e^{-\rho \tau_r}f_{r}(0,H_{\tau_r}^0)\right]-\kappa \Ex\left[\int_0^{\tau_r}e^{-\rho s}f_{rr}(R_s^r,H_s^0)ds\right], \label{eq:xi-rr}
\end{align}
where $\Gamma:=\int_0^{\infty} \frac{1}{\sqrt{2 \hat{\sigma}^2 \pi s}} e^{-\frac{\hat{\mu}^2}{2 \hat{\sigma}^2} s} d s$ is a positive constant. It follows from a direct calculation that, for all $(r,h)\in\R_+^2$,
\begin{align*}
f_r(r,h)&=\beta\int_0^{\infty}\!\!\int_0^r\!\!\int_{-\infty}^y\!\! e^{-\rho s-r+x}\phi_1(s,x,y)\phi_2(s,h)dxdyds\!\!-\!\!\beta\int_0^{\infty}\!\!\int_{-\infty}^r e^{-\rho s-r+x}\phi_1(s,x,r)\phi_2(s,h)dxds,\nonumber\\
f_{rr}(r,h)&=-f_r(r,h)
+\beta\int_0^{\infty}\int_{-\infty}^r e^{-\rho s-r+x}\phi_1(s,x,r)\phi_2(s,h)dxds\nonumber\\
&\quad-\beta\int_0^{\infty}\int_{-\infty}^r e^{-\rho s-r+x}\frac{\partial \phi_1}{\partial y}(s,x,r)\phi_2(s,h)dxds-\beta\int_0^{\infty} e^{-\rho s}\phi_1(s,r,r)\phi_2(s,h)ds.
\end{align*}
Note that, by using \eqref{eq:d}, we have
\begin{align*}
|f_r(r,h)|&\leq \beta e^{-r}\!\! \int_0^{\infty}\!\!\!\int_0^{\infty}\!\!\int_{-\infty}^y\!\! e^{-s+x}\phi_1(s,x,y)\phi_2(s,h)dxdyds+ \beta e^{-r}\!\!\int_0^{\infty}\!\!\!\int_{-\infty}^{\infty}\!\! e^{-s+x}\phi_1(s,x,r)\phi_2(s,h)dxds,\nonumber\\
&|f_r(r,h)+f_{rr}(r,h)|\leq \beta e^{-r} \int_0^{\infty}\int_{-\infty}^r e^{-s+x}\phi_1(s,x,r)\phi_2(s,h)dxds\nonumber\\
&\quad+\beta e^{-r}\int_0^{\infty}\int_{-\infty}^{\infty} e^{-s+x}\frac{\partial \phi_1}{\partial y}(s,x,r)\phi_2(s,h)dxds+\beta e^{-r}\int_0^{\infty} e^{-s+r}\phi_1(s,r,r)\phi_2(s,h)ds.
\end{align*}
In a similar fashion of \eqref{op-theta-4}-\eqref{eqM1}, we deuce that $|f_r(r,h)|\leq M_2 \beta e^{-r}$ and  $|f_r(r,h)+f_{rr}(r,h)|\leq M_3 \beta e^{-r}$, where the finite positive constants are given by
{\small\begin{align}
M_2&:=\sup_{h\in\R_+}\int_0^{\infty}\!\!\int_0^{\infty}\!\!\int_{-\infty}^y\!\! e^{-s+x}\phi_1(s,x,y)\phi_2(s,h)dxdyds+\!\! \sup_{(r,h)\in\R_+^2}\!\!\int_0^{\infty}\!\!\int_{-\infty}^r\!\! e^{- s+x}\phi_1(s,x,r)\phi_2(s,h)dxds<+\infty,\label{eq:M2}\\
M_3&:=\sup_{(r,h)\in\R_+}\int_0^{\infty}\int_{-\infty}^r e^{-s+x}\phi_1(s,x,r)\phi_2(s,h)dxds+ \sup_{(r,h)\in\R_+^2}\int_0^{\infty}\int_{-\infty}^{\infty} e^{-s+x}\frac{\partial \phi_1}{\partial y}(s,x,r)\phi_2(s,h)dxds\nonumber\\
&\quad+\sup_{(r,h)\in\R_+^2}\int_0^{\infty} e^{-s+r}\phi_1(s,r,r)\phi_2(s,h)ds<+\infty.\label{eq:M3}
\end{align}}
Therefore, it holds that
\begin{align}
(\xi_{r}+\xi_{rr})(r,0)&\leq |\kappa| \Ex\left[\int_0^{\tau_r}e^{-\rho s}|f_r+f_{rr}|(R_s^r,H_s^0)ds\right]+\beta |\kappa|\Gamma M_2 \Ex\left[e^{-\rho \tau_r}\right]\nonumber\\
&\leq |\kappa| \beta M_3 \Ex\left[\int_0^{\tau_r}e^{-\rho s-R_s^r}ds\right]+\beta |\kappa|\Gamma M_2 \Ex\left[e^{-\rho \tau_r}\right]\nonumber\\
&=  \beta e^{-r} M_3 |\kappa| \int_0^{\infty}\int_0^r\int_{-\infty}^y e^{-\rho s-x}\phi_1(s,x,y)dxdyds+\beta |\kappa|\Gamma M_2 \Ex\left[e^{-\rho \tau_r}\right]\nonumber\\
&=  \beta e^{-r} M_3M_2 |\kappa|+\beta |\kappa|\Gamma M_2 e^{-r}\leq  M_4\beta e^{-r},
\end{align}
where the positive constant $M_4$ is defined by
\begin{align}\label{eq:M4}
M_4:=M_2(M_3|\kappa|+ |\kappa|\Gamma).
\end{align}
Thus, we deduce from \eqref{eq:xi-r}-\eqref{eq:M4} that
\begin{align}\label{op-theta-m2}
\frac{1}{y^*(x,h,z)}(\psi_{rr}+\psi_r)\left(-\ln\frac{y^*(x,h,z)}{\beta},0\right)\leq (M_4+M_1)(1+x).
\end{align}
On the other hand, using Lemma~\ref{lem:v1stpart}, we have
\begin{align}
l_{rr}\left(r,z\right)&=C_1 \beta^{-\frac{p}{1-p}}\left(\frac{p}{1-p}\right)^2e^{\frac{p}{1-p}r}+C_2\beta e^{-r}+z(\beta e^{-r}-\beta\ell e^{-\ell r})\nonumber\\
&= \frac{p}{1-p}l_r(r)+\frac{\beta e^{-r}}{1-p}C_2+\frac{1}{1-p}z\beta e^{-r}-z\beta e^{-\ell r}\left(\ell+\frac{p}{1-p}\right) .
\end{align}
Then, by using the condition $\rho>\frac{\alpha^2 p+1}{2(1-p)}$, we deduce that
\begin{align}\label{op-theta-l}
&\frac{1}{y^*(x,h,z)}l_{rr}\left(-\ln\frac{y^*(x,h,z)}{\beta},z\right)\nonumber\\
&\leq \frac{1}{y^*(x,h,z)}\left[\frac{1}{1-p}(l_r+\varphi_r)\left(-\ln\frac{y^*(x,h,z)}{\beta},h,z\right)+\frac{ y^*(x,h,z)}{1-p}C_2+\frac{1}{1-p}zy^*(x,h,z)\right]\nonumber\\
&=\frac{1}{1-p}x+\frac{1}{1-p}C_2+\frac{1}{1-p}z\leq \frac{1}{1-p}(x+z)+2(1-p)\beta^{-\frac{1}{1-p}}.
\end{align}
By using Lemma~\ref{lem:v1stpart} again, we have $|zl_{rz}(r,z)|=z\beta (e^{-\ell r}-e^{-r})\leq l_r(r,z)$ for all $(r,z)\in\R_+^2$. Thus, it holds that
\begin{align}\label{eq:op-theta-rz}
\frac{1}{y^*(x,h,z)}\left|zl_{rz}\left(-\ln\frac{y^*(x,h,z)}{\beta},z\right)\right|&\leq\frac{1}{y^*(x,h,z)}(l_r+\varphi_r)\left(-\ln\frac{y^*(x,h,z)}{\beta},h,z\right)=x.
\end{align}
Moreover, note that
\begin{align}\label{op-theta-rh}
\left|\psi_{rh}\left(r,z\right)\right|&=\beta \Ex\left[e^{-\rho \zeta_{h}-R_{\zeta_{h}}^{r}}\mathbf{1}_{\zeta_h<\tau_r}\right]\leq \beta \Ex\left[\exp\left(-\rho \zeta_{h}-r-\left(\frac{\alpha^2}{2}-\rho\right)\zeta_{h}-\alpha B^1_{\zeta_{h}}- L_{\zeta_{h}}^{r}\right)\right]\nonumber\\
&\leq \beta e^{-r}\Ex\left[\exp\left(-\frac{\alpha^2}{2}\zeta_{h}-\alpha B^1_{\zeta_{h}}\right)\right]=\beta e^{-r}.
\end{align}
In lieu of \eqref{op-theta},  \eqref{eq:psi-varphi}, \eqref{op-theta-m}, \eqref{op-theta-l} \eqref{eq:op-theta-rz} and \eqref{op-theta-rh}, we deduce  $|\theta^*(x,h,z)|\leq C_o(1+x+z)$, where the positive constant $C_o$ is defined by
\begin{align}\label{Co}
C_o:=1+\left(\frac{1}{1-p}+2(1-p)\beta^{-\frac{1}{1-p}}+M_1+M_4\right)\gamma_1+\gamma_2.
\end{align}
Here, the constant $C_2$ is given by \eqref{eq:C1C2} and constants $M_1$ and $M_4$ are defined as \eqref{eqM1} and \eqref{eq:M4}.

Next, we show the linear growth of $c^*(x,h,z)$ on $(x,h,z)\in\R_+^3$. Note that $y^*(x,h,z)=u_{x}(x,h,z)$ for all $(x,h,z)\in\R_+^3$, we arrive at
\begin{align}\label{ineq:c}
|c^*(x,h,z)|=u_{x}(x,h,z)^{\frac{1}{p-1}}=\left(\frac{1}{y^*(x,h,z)}\right)^{\frac{1}{1-p}}.
\end{align}
Using the relationship $-x=v_y(y^*(x,h,z),h,z)$, we can see that
\begin{align}\label{op-theta-y1}
x=-v_y(y^*(x,z),h,z)&=\frac{1}{y^*(x,h,z)}(l_r+\psi_r)\left(-\ln\frac{y^*(x,h,z)}{\beta},h,z\right)\nonumber\\
&\geq \frac{1}{y^*(x,h,z)}l_r\left(-\ln\frac{y^*(x,h,z)}{\beta},z\right).
\end{align}
Thus, Lemma \ref{lem:v1stpart} yields that, for all $r\in [1, \infty)$,
\begin{align*}
l_r(r,z)&=\beta^{-\frac{p}{1-p}}\frac{pC_1}{1-p}e^{\frac{pr}{1-p}}-C_2\beta e^{-r}+z\beta(e^{-\ell r}-e^{-r})\geq(\beta e^{-r})^{-\frac{p}{1-p}}\left(\frac{pC_1}{1-p}- C_2\beta^{\frac{1}{1-p}} e^{-\frac{r}{1-p}}\right)\nonumber\\
&= (\beta e^{-r})^{-\frac{p}{1-p}}\frac{2(1-p)}{2\rho(1-p)-\alpha^2p}\left(1- e^{-\frac{1}{1-p}r}\right)\geq (\beta e^{-r})^{-\frac{p}{1-p}}\frac{2(1-p)}{2\rho(1-p)-\alpha^2p}\left(1- e^{-\frac{1}{1-p}}\right).
\end{align*}
This implies that, for the case in which $y^*(x,h,z)\leq \beta e^{-1}$,
\begin{align}\label{op-theta-y2}
l_r\left(-\ln\frac{y^*(x,h,z)}{\beta},z\right)\geq \frac{2(1-p)}{2\rho(1-p)-\alpha^2 p}\left(1- e^{-\frac{1}{1-p}}\right)\left(\frac{1}{y^*(x,h,z)}\right)^{\frac{p}{1-p}}.
\end{align}
For the case with $y^*>\beta e^{-1}$, we have $0< y^*(x,h,z)^{-\frac{1}{1-p}}\leq \beta^{-\frac{1}{1-p}} e^{\frac{1}{1-p}}$ for all $(x,h,z)\in\R_+^3$. Hence, $(x,h,z)\to y^*(x,h,z)^{-\frac{1}{1-p}}$ is bounded. Thus, it follows from \eqref{op-theta-y1} and \eqref{op-theta-y2} that, for all $(x,h,z)\in  \R_+^3$,
\begin{align}\label{op-y}
0<\left(\frac{1}{y^*(x,h,z)}\right)^{\frac{1}{1-p}}\leq C_{q}(1+x),
\end{align}
where the positive constant $C_q$ is specified as
\begin{align}\label{eq:Cq}
C_q:=\beta^{-\frac{1}{1-p}} e^{\frac{1}{1-p}}+\frac{(2\rho(1-p)-\alpha^2 p)}{2(1-p)}\left(1- e^{-\frac{1}{1-p}}\right)^{-1}.
\end{align}
We can then conclude from \eqref{ineq:c} and \eqref{op-y} that $|c^*(x,z)|\leq C_{q}(1+x)$ for $(x,h,z)\in  \R_+^3$. Hence, it follows from \eqref{eq:op-control-growth} that, for any $T\in\R_+$, the SDE \eqref{eq:optimal-SDE-XZ} satisfied by $X^*$ admits a weak solution on $[0,T]$ (c.f. \cite{Laukajtys13}), which gives that $(\theta^*,c^*)\in\mathbb{U}^{\rm r}$.

On the other hand, fix $(T,x,h,z)\in\R_+^4$ and $(\theta,c)=(\theta_t,c_t)_{t\geq0}\in\mathbb{U}^{\rm r}$.
By applying It\^{o}'s formula to $e^{-\rho T}u(X_{T},I_T,Z_{T})$, we arrive at
\begin{align}\label{eq:itoveri}
&e^{-\rho T}u(X_{T},I_T,Z_{T})+\int_0^{T} e^{-\rho s} \frac{(c_s)^p}{p}ds\nonumber\\
&\quad=u(x,h,z)+\int_0^{T}e^{-\rho s} u_{x}(X_{s},I_s,Z_{s})\theta_s^{\top}\sigma dW_s+\int_0^{T}e^{-\rho s}\sigma_B u_{z}(X_{s},I_s,Z_{s})dW^{\gamma}_s\nonumber\\
&\qquad+\int_0^{T}e^{-\rho s}\sigma_Z Z_s u_{z}(X_{s},I_s,Z_{s})dW^{\eta}_s+\int_0^{T} e^{-\rho s}(u_z-u_x)(X_{s},I_s,Z_{s})dm_s\nonumber\\
&\qquad+\int_0^{T} e^{-\rho s} u_x(X_{s},I_s,Z_{s})dL_s^X+\int_0^{T} e^{-\rho s}({\cal L}^{\theta_s,c_s} u-\rho u)(X_{s},I_s,Z_{s})ds,
\end{align}
where the operator ${\cal L}^{\theta,c}$ with $(\theta,c)\in \R^d\times \R_+$ is defined on $C^{2}(\R_+^2)$ that
\begin{align*}
{\cal L}^{\theta,c}g&:=\theta^{\top}\mu g_x+\frac{1}{2}\theta^{\top}\sigma\sigma^{\top}\theta g_{xx}+\sigma_Z \theta^{\top}\sigma\eta z(g_{xx}-g_{xz})-\sigma_B \theta^{\top}\sigma\gamma g_{xh}+\frac{c^p}{p}-cg_x+\frac{1}{2}\sigma_B^2g_{hh} -\mu_B g_h\\
&\quad+\frac{1}{2}\sigma_Z^2z^2(g_{zz}+g_{xx}-2g_{xz})+\mu_Z z(g_z-g_x)+\sigma_Z\sigma_B z\eta^{\top}\gamma (g_{xh}-g_{hz}),\quad \forall g\in C^{2}(\R_+^2).
\end{align*}
Taking the expectation on both sides of the equality \eqref{eq:itoveri}, we deduce from the Neumann boundary condition $u_x(0,h,z)=\beta$ and $u_h(x,0,z)=u_x(x,0,z)$ that
\begin{align}\label{eq:value-ineq}
&\Ex\left[\int_0^{T} e^{-\rho s} \frac{(c_s)^p}{p}ds- \beta \int_0^{T} e^{-\rho s}dL_s^X \right]=u(x,h,z)-\Ex\left[e^{-\rho T}u(X_{T},I_T,Z_{T})\right]\nonumber\\
&\quad+\Ex\left[\int_0^{T} e^{-\rho s}({\cal L}^{\theta_s,c_s} u-\rho u)(X_{s},I_s,Z_{s})ds\right]\leq u(x,h,z)-\Ex\left[e^{-\rho T}u\left(X_{T},I_T,Z_{T}\right)\right].
\end{align}
Here, the last inequality in \eqref{eq:value-ineq} holds true due to $({\cal L}^{\theta,c} u-\rho u)(x,h,z)\leq 0$ for all $(x,h,z)\in\R_+^3$ and $(\theta,c)\in\R^d\times\R_+$. We next verify the validity of the so-called transversality conditions:
\begin{align}
\limsup_{T\to \infty}\Ex\left[e^{-\rho T}u(X_T,I_T,Z_T)\right]\geq 0,\label{eq:transcond-1}\\
\lim_{T\to \infty}\Ex\left[e^{-\rho T}u(X_T^*,I_T,Z_T)\right]=0.\label{eq:transcond}
\end{align}
In view of Lemma \ref{lem:v1stpart} and Proposition \ref{prop:secondtermm}, it follows that $x\to u(x,h,z)$, $h\to u(x,h,z)$ and $z\to u(x,h,z)$ are non-decreasing. Thus, we get
\begin{align}\label{eq:limsup}
\limsup_{T\to \infty}\Ex\left[e^{-\rho T}u(X_T,I_T,Z_T)\right]\geq \limsup_{T\to \infty}\Ex\left[e^{-\rho T}u(0,0,0)\right]=0.
\end{align}
Using Lemma \ref{lem:v1stpart} and Proposition \ref{prop:secondtermm} again, it holds that $|u_x(x,h,z)|\leq \beta$, $|u_h(x,h,z)|\leq \beta$ and $|u_z(x,h,z)|\leq \frac{\beta}{\ell}$ for all $(x,h,z)\in \R_+^3$. Thus, we can see that, for all $(x,z,h) \in \R_+^3$,
\begin{align}\label{eq:dominate-u}
|u(x,h,z)|&\leq |u(x,h,z)-u(x,h,0)|+|u(x,h,0)-u(x,0,0)|+|u(x,0,0)-u(0,0,0)|+|u(0,0,0)|\nonumber\\
&\leq \beta\left(x+h\right)+\frac{\beta}{\ell}z+|u(0,0,0)|.
\end{align}
By applying It{\^o}'s formula to $|I_t|^2$ and $|X_t|^2$, it follows from \eqref{eq:op-control-growth} and the Gronwall's lemma that, for all $t\geq0$,
\begin{align}
\Ex[|I_t|^2]&\leq h^2+(\sigma_B^2+\mu_B^2) te^t,\label{ineq-It}\\
\Ex\left[|X_t^*|^2\right]&\leq  x^2+2Kt\left[1+Ke^{Kt}+z^2 e^{(2|\mu_Z|+3\sigma_Z^2)t}\left(1+\frac{Ke^{Kt}}{2|\mu_Z|+3\sigma_Z^2}\right)\right](1+x^2),\label{ineq-Xt}
\end{align}
where the positive constant $K$ is specified as
\begin{align}\label{eq:constK}
K:=4 C_o|\mu|+2C_o^2|\sigma\sigma^{\top}|+|\mu_Z|+\sigma_Z^2+4C_o|\sigma_Z\sigma\eta|.
\end{align}
Let us define the constant
\begin{align}\label{boundrho}
\rho_0:=\frac{\alpha^2 |p|+1}{2(1-p)}+K+2|\mu_Z|+3\sigma_Z^2+1
\end{align}
with $K$ being given in \eqref{eq:constK}. Then, using estimates \eqref{eq:dominate-u}, \eqref{ineq-It} and \eqref{ineq-Xt}, it follows that, for the discount rate $\rho>\rho_0$,
\begin{align*}
\lim_{T\to \infty}\Ex\left[e^{-\rho T}|u(X_T^{*},I_T,Z_T)|\right]\leq \beta\lim_{T\to \infty}\Ex\left[e^{-\rho T}(X_T^{*}+I_T+Z_T)\right]=0.
\end{align*}
Finally, letting $T\to\infty$ in \eqref{eq:value-ineq}, we obtain from \eqref{eq:dominate-u} and DCT that, for any $(\theta,c)\in\mathbb{U}^{\rm r}$,
\begin{align*}
\Ex\left[\int_0^{\infty} e^{-\rho s} \frac{(c_s)^p}{p}ds- \beta \int_0^{\infty} e^{-\rho s}dL_s^X \right]\leq u(x,h,z),\quad \text{for all}~ (x,h,z)\in\R_+^3,
\end{align*}
where the equality holds when $(\theta,c)=(\theta^*,c^*)$. Thus, the proof of the theorem is complete.
\end{proof}

\begin{remark}\label{rem:rho0}
To ensure the validity of the transversality conditions \eqref{eq:transcond-1} and \eqref{eq:transcond}, we assume that the discount rate $\rho>\rho_0$, where the constant $\rho_0>0$ only depends on model parameters $(\mu,\sigma,\mu_B,\sigma_B,\mu_Z,\sigma_Z,\gamma,p,\beta)$, and moreover, it is explicitly specified in \eqref{boundrho}. For example, let us consider the model parameters specified as $d=n=1$, $\mu=0,1$, $\sigma=1$, $\mu_B=0.1$, $\sigma_B=0.1$, $\mu_Z=0.1$, $\sigma_Z=0.1$, $p=0.5$, $\gamma=1$ and $\beta=1$. By a direct calculation, $\rho_0\approx 2.63$, thus the condition on the discount rate $\rho$ requires that $\rho>2.63$.
\end{remark}

\begin{remark}\label{rem:XZ}
In fact, the state processes of the primal control problem \eqref{eq_prob_IBP} and the auxiliary control problem \eqref{eq_mfg-2} satisfy the following relationship:
\begin{align}
X_t&=x+V_t^{\theta,c}-(m_t+Z_t-m_0-z)+\sup_{s\in[0,t]}\left(-x-V_s^{\theta,c}+(m_s+Z_s-m_0-z)\right)^+,\label{eq:state-1}\\
I_t&=h+(m_t-m_0)-B_t,\quad \forall t\geq0.\label{eq:state-2}
\end{align}
Therefore, we can obtain the auxiliary state process $(X,I,Z)=(X_t,I_t,Z_t)_{t\geq 0}$ by using the process $(V^{\theta,c},B,m,Z)=(V_t^{\theta,c},B_t,m_t,Z_t)_{t\geq 0}$. However, from \eqref{eq:state-1} and \eqref{eq:state-2}, we can also see that different primal state processes $(V^{\theta,c},B,m,Z)$ may correspond to the same auxiliary state process $(X,I,Z)$. Theorem \ref{thm:verification} gives the optimal feedback control $(\theta^*,c^*)$ in terms of $(X,I,Z)$ but not by $(V^{\theta,c},B,m,Z)$. This is an important reason why we introduce the auxiliary state process $(X,I,Z)$ and study the auxiliary optimal control problem instead, which allows us to characterize the optimal control $(\theta^*,c^*)$ in the feedback form.
\end{remark}

The following lemma shows that the expectation of the total optimal capital injection is always positive and finite.
\begin{lemma}\label{lem:inject-captial}
Consider the optimal investment-consumption strategy $(\theta^*,c^*)=(\theta_t^*,c_t^*)_{t\geq 0}$ provided in Theorem \ref{thm:verification}. Then, we have
\begin{itemize}
\item [{\rm(i)}] The  expectation of the discounted capital injection under the optimal strategy $(\theta^*,c^*)$ is finite. Namely, for $\rho>\rho_0$ with $\rho_0>0$ being given in Theorem \ref{thm:verification},
\begin{align}\label{eq:dAstarfinite-1}
\Ex\left[\int_0^{\infty} e^{-\rho t}dA^*_t \right]<+\infty.
\end{align}
\item[{\rm(ii)}] The  expectation of the discounted capital injection under the optimal strategy $(\theta^*,c^*)$ is positive. Namely, for $\rho>\rho_0$ with $\rho_0>0$ being given in Theorem \ref{thm:verification}, it holds that
\begin{align}\label{eq:dAstarfinite}
\Ex\left[\int_0^{\infty} e^{-\rho t}dA^*_t \right]\geq z\frac{1-\kappa}{\kappa}\left(1+\frac{x}{z}\right)^\frac{\kappa}{\kappa-1}>0.
\end{align}
\end{itemize}
 Here, the optimal capital injection under the optimal strategy $(\theta^*,c^*)$ is given by
\begin{align}\label{A-sing2}
A^*_t =0\vee \sup_{s\leq t}(M_{s}-V_{s}^{\theta^*,c^*}), \quad\forall t\geq 0.
\end{align}
\end{lemma}

\begin{proof}
We first prove the item (i). For $(\mathrm{v},\textrm{m},z,b)\in\R_+^3\times\R$, we have from \eqref{eq_prob_IBP} that
\begin{align}\label{eq:discountAsatr}
 \beta\Ex\left[\int_0^{\infty} e^{-\rho t}dA^*_t \right]= \Ex\left[\int_0^{\infty} e^{-\rho t} \frac{(c^*_t)^p}{p}dt\right] -{\rm w}(\mathrm{v},\textrm{m},z,b).
\end{align}
Thus,  to prove \eqref{eq:dAstarfinite-1}, it suffices to show that
\begin{align}\label{eq:disountUcstar}
\Ex\left[\int_0^{\infty} e^{-\rho t} \frac{(c^*_t)^p}{p}dt\right]<+\infty.
\end{align}
By \eqref{eq:discountAsatr}, $\Ex\left[\int_0^{\infty} e^{-\rho t} \frac{(c^*_t)^p}{p}dt\right]$ can not be $-\infty$ because ${\rm w}(\mathrm{v},\textrm{m},z,b)$ is finite and $\Ex\left[\int_0^{\infty} e^{-\rho t}dA^*_t \right]$ is nonnegative. The estimate \eqref{eq:disountUcstar} obviously holds for the case $p<0$ as $\Ex\left[\int_0^{\infty} e^{-\rho t} \frac{(c^*_t)^p}{p}dt\right]$ is negative in this case. Hence, we only focus on the case with $p\in(0,1)$. For $p\in(0,1)$, it follows from  \eqref{eq:op-control-growth} and \eqref{ineq-Xt} that
\begin{align*}
&\Ex\left[\int_0^{\infty} e^{-\rho t} \frac{(c_t^*)^p}{p}dt\right]\leq\frac{1}{p}(C_q)^p\Ex\left[\int_0^{\infty} e^{-\rho t} (1+|X^*_t|)^p dt\right]\nonumber\\
&\leq \frac{(C_q)^p}{p}\int_0^{\infty} e^{-\rho t} (1+\Ex\left[|X^*_t|\right]) dt\leq \frac{(C_q)^p}{p}\int_0^{\infty} e^{-\rho t} [1+(1+\Ex[|X^*_t|^2])] dt\nonumber\\
&\leq \frac{(C_q)^p}{p}\int_0^{\infty} e^{-\rho t} \left\{2+x^2+2Kt\left[1+Ke^{Kt}+z^2 e^{(2|\mu_Z|+3\sigma_Z^2)t}\left(1+\frac{Ke^{Kt}}{2|\mu_Z|+3\sigma_Z^2}\right)\right](1+x^2)\right\}dt\nonumber\\
&\leq K(1+x^2)<+\infty,
\end{align*}
where $x=({\rm v}-\textrm{m}\vee b-z)^+$ via \eqref{eq:wu}, $C_q>0$ is the constant given by \eqref{eq:Cq}, and $K>0$ is a constant depending on model parameters $(\mu,\sigma,\mu_Z,\sigma_Z,\mu_B,\sigma_B,p,\beta,\rho)$ only.

Next, we prove the item (ii). For any admissible portfolio $\theta=(\theta_t)_{t\geq0}$, we introduce, for all $t\in\R_+$,
\begin{align}\label{eq:wealth-theta}
\tilde{V}_t^{\theta} &=\textrm{v}+\int_0^t\theta_s^{\top}\mu ds+\int_0^t \theta_s^{\top}\sigma dW_s,\quad
\tilde{A}_t^{\theta}=0\vee \sup_{s\leq t}(Z_{s}-\tilde{V}^{\theta}_{s}).
\end{align}
Note that $c^*_t>0$ for all $t\in\R_+$. Then, it follows from \eqref{eq:wealth2}, \eqref{A-sing2} and \eqref{eq:wealth-theta} that $\tilde{V}_t^{\theta^*}\geq V^{\theta^*,c^*}_t$ for all $t\in\R_+$, and hence
\begin{align}\label{eq:tilde-w}
\Ex\left[\int_0^{\infty} e^{-\rho t}dA^*_t \right]>\Ex\left[\int_0^{\infty} e^{-\rho t}d\tilde{A}_t^{\theta^*} \right]\geq \inf_{\theta}\Ex\left[\int_0^{\infty} e^{-\rho t}d\tilde{A}_t^{\theta} \right]=:\tilde{\mathrm{w}}(\mathrm{v},z).
\end{align}
It is not difficult to verify that, for all $(\mathrm{v},z)\in\R_+\times(0,\infty)$,
\begin{align}\label{eq:tilde-v}
\tilde{\mathrm{w}}(\mathrm{v},z)=z\frac{1-\ell}{\ell}\left(1+\frac{(\mathrm{v}-z)^+}{z}\right)^\frac{\ell}{\ell-1}>0,
\end{align}
where the constant $\ell \in(0,1)$ is given by \eqref{eq:kappa}. Thus, we deduce from \eqref{eq:tilde-w} and \eqref{eq:tilde-v} that
\begin{align}\label{eq:low-injection}
\Ex\left[\int_0^{\infty} e^{-\rho t}dA^*_t \right]\geq z\frac{1-\ell}{\ell}\left(1+\frac{(\mathrm{v}-z)^+}{z}\right)^\frac{\kappa}{\kappa-1}>0,
\end{align}
which completes the proof.
\end{proof}

\vspace{0.2in}
\noindent
\textbf{Acknowledgements.}\quad We sincerely thank two referees for their constructive comments and suggestions. L. Bo and Y. Huang are supported by Natural Science Basic Research Program of Shaanxi under grant no. 2023-JC-JQ-05, National Natural Science Foundation of China under grant no. 11971368 and Fundamental Research Funds for the Central Universities under grant no. 20199235177. X. Yu is supported by the Hong Kong RGC General Research Fund (GRF) under grant no. 15304122.

\begin{appendix}

\vspace{0.3cm}

\section{Appendix}\label{appendix}

This section provides the proof of Lemma \ref{lem:propoertyu} and a sketch of the proof of Lemma 
\ref{lem:v1stpart}.

\begin{proof}[Proof of Lemma~\ref{lem:propoertyu}] Let us first fix $(h,z)\in\R_+^2$. For any $\epsilon>0$, denote by $(\theta^{\epsilon}(x,z),c^{\epsilon}(x,z))$ the $\epsilon$-optimal control strategy for \eqref{eq_mfg-2}. Namely, for $x\in\R_+$,
\begin{align}\label{epsiloncontrol}
u(x,h,z)\leq J(x,z;\theta^{\epsilon}(x,h,z),c^{\epsilon}(x,h,z))+\epsilon.
\end{align}
Then, for any $x_1>x_2\geq0$, we have from \eqref{epsiloncontrol} that
\begin{align}\label{eq:diffVy1y2}
u(x_1,h,z)-u(x_2,h,z) &\geq  J(x_1,h,z;\theta^{\epsilon}(x_2,h,z),c^{\epsilon}(x_2,h,z))-J(x_2,h,z;\theta^{\epsilon}(x_2,h,z),c^{\epsilon}(x_2,h,z))-\epsilon\nonumber\\
&=-\beta\Ex\left[\int_0^{\infty}  e^{-\rho s}d(L_s^{x_1}-L_s^{x_2})\right]-\epsilon,
\end{align}
where $L_s^{x}$ for $s\in\R_+$ is the local time process with $X_0^{x}=x$. Thus, integration by parts yields that, for all $T\geq 0$,
\begin{align*}
\int_0^{T} e^{-\rho s}dL_s^{x} = e^{-\rho T}L_T^{x} + \rho\int_0^{T}  L_s^{x}e^{-\rho s}ds.
\end{align*}
Using the solution representation of ``the Skorokhod problem", it follows that, for all $s\in\R_+$,
\begin{align*}
L_s^{x}=\sup_{t\in[0,s]}\left(x+\int_0^{t}\theta_{r}^{\top}\mu dr+\int_0^{t}\theta_{r}^{\top}\sigma dW_{r}  -\int_0^{t} c_{r} dr -\int_0^{t} \mu_Z Z_rdr-\int_0^t \sigma_Z Z_rdW^{\eta}_r-\int_0^{t} dm_{r}\right)^-.
\end{align*}
By this, we have $x\to L_s^{x}$ is non-increasing. Moreover,  it holds that, $\Px$-a.s.
\begin{align}\label{eq:LiphatL0}
\sup_{s\geq 0}\left|L_s^{x_1}-L_s^{x_2}\right|\leq |x_1-x_2|.
\end{align}
Using the fact $L_s^{x_1}-L_s^{x_2}\leq0$ whenever $x_1>x_2\geq0$ and MCT, it follows that, for all $s\geq 0$,
\begin{align}\label{eq:diffLref}
&\Ex\left[\int_0^{\infty} e^{-\rho s}d(L_s^{x_1}-L_s^{x_2})\right]=\Ex\left[\int_0^{\infty} e^{-\rho s}dL_s^{x_1}\right]-\Ex\left[\int_0^{\infty} e^{-\rho s}dL_s^{x_2}\right]\nonumber\\
&\qquad=\lim_{T\to \infty}\Ex\left[\int_0^{T} e^{-\rho s}dL_s^{x_1}\right]-\lim_{T\to \infty}\Ex\left[\int_0^{T} e^{-\rho s}dL_s^{x_2}\right]\nonumber\\
&\qquad=\lim_{T\to\infty}\left\{\Ex\left[e^{-\rho T}(L_T^{x_1}-L_T^{x_2})\right]+ \rho \Ex\left[\int_0^T e^{-\rho s}(L_s^{x_1}-L_s^{x_2})ds\right]\right\}\leq 0.
\end{align}
Hence, we have from \eqref{eq:diffVy1y2} that $u(x_1,h,z)-u(x_2,h,z)\geq -\epsilon$. Since $\epsilon>0$ is arbitrary, we get $u(x_1,h,z)\geq u(x_2,h,z)$. This conclude that $x\to u(x,h,z)$ is non-decreasing. On the other hand, it follows from \eqref{eq:LiphatL0}, \eqref{eq:diffLref} and MCT that
\begin{align}\label{eq:betaLip}
\left|u(x_1,h,z)-u(x_2,h,z)\right|&\leq \beta \sup_{(\theta, c)\in\mathbb{U}^{\rm r}}\Ex\left[\int_0^{\infty} e^{-\rho s}d(L_s^{x_2}-L_s^{x_1})\right]\nonumber\\
&=\beta \sup_{(\theta, c)\in\mathbb{U}^{\rm r}} \lim_{T\to\infty}\Ex\left[e^{-\rho T}(L_T^{x_2}-L_T^{x_1}) + \rho \int_0^Te^{-\rho s} (L_s^{x_2}-L_s^{x_1})ds\right]\nonumber\\
&\leq \beta \lim_{T\to\infty}\left(e^{-\rho T}|x_1-x_2| + \rho|x_1-x_2| \int_0^T e^{-\rho s}ds\right)\nonumber\\
&=\beta \lim_{T\to\infty}\left(e^{-\rho T}+ \rho\int_0^T e^{-\rho s}ds\right)|x_1-x_2|=\beta|x_1-x_2|.
\end{align}

Next, we fix $(x,z)\in\R_+^2$. Let $m^h=(m_s^h)_{s\geq0}$ and $L^{x,h}=(L_s^{x,h})_{s\geq0}$ be the respective local time process of $I$ and $X$ with $m_0^h=h\in\R_+$ and $X_0^{x,z}=x\in\R_+$. Using the solution representation of ``the Skorokhod problem" again, we can obtain that, for all $s\geq 0$,
\begin{align*}
\begin{cases}
\displaystyle m_s^{h}=\sup_{\ell\in[0,s]}\left(h-\int_0^{\ell} \mu_Bds-\int_0^{\ell}\sigma_B dW^{\gamma}_s\right)^-,\\[0.9em]
\displaystyle L_s^{x,h}=\sup_{t\in[0,s]}\left(x+\int_0^{t}\theta_{r}^{\top}\mu dr+\int_0^{t}\theta_{r}^{\top}\sigma dW_{r}  -\int_0^{t} c_{r} dr -\int_0^{t} \mu_Z Z_rdr-\int_0^t \sigma_Z Z_rdW^{\eta}_r- m_{t}^h\right)^-.
\end{cases}
\end{align*}
This implies that both $z\to m_s^{h}$ and $z\to L_s^{x,h}$ are non-increasing. Moreover, for $h_1,h_2\geq 0$, it holds that, $\Px$-a.s.
\begin{align}\label{eq:LiphatL}
\sup_{s\geq 0}\left|L_s^{x,h_1}-L_s^{x,h_2}\right|\leq \sup_{s\geq 0}\left|m_s^{h_1}-m_s^{h_2}\right|\leq |h_1-h_2|.
\end{align}
Then, in a similar fashion, we can also show that $h\to u(x,h)$ is also non-decreasing, and it holds that, for all $(h_1,h_2)\in\R_+^2$,
\begin{align}\label{eq:betaLip2}
|u(x,h_1,z)-u(x,h_2,z)|\leq \beta |h_1-h_2|.
\end{align}
Finally, fix $(x,h)\in\R_+^2$, by applying the argument to $z \to u(x,h,z)$, we can obtain that for all $(z_1,z_2)\in\R_+$,
\begin{align}\label{eq:betaLip3}
|u(x,h,z_1)-u(x,h,z_2)|\leq \beta \left(\sigma_Z^2+\frac{|\mu_Z|}{\rho-\mu_Z}+\frac{3}{\rho-2\mu_Z-\sigma_Z^2}\right) |z_1-z_2|.
\end{align}
Therefore, we deduce from \eqref{eq:betaLip}, \eqref{eq:betaLip2} and \eqref{eq:betaLip3} that
\begin{align*}
&|u(x_1,h_1,z_1)-u(x_2,h_2,z_2)|\\
&\leq |u(x_1,h_1,z_1)-u(x_2,h_1,z_1)|+|u(x_2,h_1,z_1)-u(x_2,h_2,z_1)|+|u(x_2,h_2,z_1)-u(x_2,h_2,z_2)\\
&\leq \beta (|x_1-x_2|+|h_1-h_2|)+\beta \left(\sigma_Z^2+\frac{|\mu_Z|}{\rho-\mu_Z}+\frac{3}{\rho-2\mu_Z-\sigma_Z^2}\right) |z_1-z_2|.
\end{align*}
Thus, we complete the proof of the lemma.
\end{proof}

\begin{proof}[Sketch of Proof of Lemma~\ref{lem:v1stpart}]
In a similar fashion of the proof of Lemma \ref{lem:secondtermm} and Lemma \ref{lem:derivative-m}, we can prove that the function $l(r,z)$ given by \eqref{eq:fcnl} is a classical solution to the  Neumann problem \eqref{eq:HJB-l}. Moreover, if the Neumann problem \eqref{eq:HJB-l} has a classical solution $l(r,z)$ for $r\in\R_+$ satisfying $|l(r,z)|\leq C(1+e^{qr}+z^q)$ for some $q>1$ and a constant $C>0$ depending on $(\mu,\sigma,\mu_Z,\sigma_Z,p)$, then this solution $l(r,z)$ admits the probabilistic representation \eqref{eq:fcnl}.

Next, we derive the  explicit form of the classical solution to Eq. \eqref{eq:HJB-l}. It follows from the probabilistic representation \eqref{eq:fcnl} that we consider the candidate solution admitting the form $l(r,h)=f(r)+z \psi(r)$ for Eq.~\eqref{eq:HJB-l}. In particular, the functions $r\to f(r)$ and $r\to\psi(r)$ satisfy the following equations, respectively:
\begin{align}
 -\rho f(r)+\left(\frac{\alpha}{2}-\rho\right) f_r(r)+\frac{\alpha^2}{2} f_{rr}(r)+\frac{1-p}{p}\beta^{-\frac{p}{1-p}}e^{\frac{p}{1-p}r}&=0,\label{eq:l}\\[1em]
 (\mu_Z -\rho) \psi(r)+\left(\frac{\alpha^2}{2}+\kappa_2-\rho\right) \psi_r(r)+\frac{\alpha^2}{2} \psi_{rr}(r)-(\mu_Z-\kappa_2) \beta  e^{-r}&=0.\label{eq:h}
\end{align}
By solving Eq.s \eqref{eq:l} and \eqref{eq:h}, we obtain
\begin{align*}
f(r)&=
\displaystyle C_1{\beta}^{-\frac{p}{1-p}}e^{\frac{p}{1-p}r}+C_2\beta  e^{-r}+C_3  e^{\frac{\rho}{\alpha}r},\quad
\psi(r)=\beta e^{-r}+C_4 e^{-\ell r}+C_5 e^{-\hat{\ell} r},
\end{align*}
where the constant $C_1:=\frac{2(1-p)^3}{p(2\rho(1-p)-\alpha^2 p)}$, and the constants $C_i$ with $i=2,\ldots,5$ are unknown real constants which will be determined later. Above, the constant $\ell,\tilde{\ell}$  are the roots of the quadratic equation given by
\begin{align*}
\frac{1}{2}\alpha^2 \ell^2+\left(\rho-\kappa_2-\frac{1}{2}\alpha^2\right)\ell+\mu_Z-\rho=0,
\end{align*}
which are given by
\begin{align*}
    \ell&=\frac{-(\rho-\kappa_2-\frac{1}{2}\alpha^2)+\sqrt{(\rho-\kappa_2-\frac{1}{2}\alpha^2)^2+2\alpha^2(\rho-\mu_Z)}}{\alpha^2}>0,\\
    \tilde{ \ell}&=\frac{-(\rho-\kappa_2-\frac{1}{2}\alpha^2)-\sqrt{(\rho-\kappa_2-\frac{1}{2}\alpha^2)^2+2\alpha^2(\rho-\mu_Z)}}{\alpha^2}<0.
\end{align*}
Using the probability representation \eqref{eq:fcnl}, we look for such functions $f(r)$ and $\psi(r)$ with $C_3=C_5=0$ and such that the Neumann boundary conditions $f_r(r)=0$ and $\psi_r(r)=0$ holds. This implies that $C_2=\frac{2(1-p)^2}{2\rho(1-p)-\alpha^2 p}\beta^{-\frac{1}{1-p}}$ and $C_4=- \frac{\beta}{\kappa}$. With the above specified constants $C_i$ with $i=1,\ldots,5$, we can easily verify that $l(r,z)=f(r)+z \psi(r)$  satisfies Eq. \eqref{eq:HJB-l}. Furthermore, we can verify that the solution $l(r,z)$ satisfies the growth condition $|l(r,z)|\leq C(1+e^{qr}+z^q)$ for some $q>1$ and some constant $C>0$. Then, this solution $l(r,z)$ admits the probabilistic representation \eqref{eq:fcnl}. In other words, the probabilistic representation \eqref{eq:fcnl} has the explicit form \eqref{eq:explicit-l}. Thus, we complete the proof of the lemma.
\end{proof}

\end{appendix}

\end{document}